\pdfoutput=1
\RequirePackage{ifpdf}
\ifpdf % We~are running pdfTeX in pdf mode
\documentclass[pdftex]{sigma}
\else
\documentclass{sigma}
\fi

\usepackage{mathtools}
\usepackage{amsfonts}
\usepackage{amsthm}
\usepackage{amssymb}
\usepackage{dsfont}
\usepackage{enumitem}     % managing enumerations
\usepackage{cleveref}
\usepackage{stmaryrd}
\usepackage{mathrsfs}
\usepackage{makecell}
\usepackage{enumitem}
\usepackage{breqn}
\usepackage{longtable}
\usepackage{tikz}
\usepackage{tikz-cd}
\usetikzlibrary{matrix,shapes,decorations.pathreplacing}
\usepackage{multicol}
\usepackage{float}

\numberwithin{equation}{section}

\newtheorem{Theorem}{Theorem}[section]
\newtheorem{Corollary}[Theorem]{Corollary}
\newtheorem{Lemma}[Theorem]{Lemma}
\newtheorem{Proposition}[Theorem]{Proposition}

\theoremstyle{definition}
\newtheorem{Definition}[Theorem]{Definition}
\newtheorem{Example}[Theorem]{Example}
\newtheorem{Remark}[Theorem]{Remark}

\newtheorem*{Conjecture*}{Conjecture}
\newtheorem*{Notation*}{Notations}

\newcommand{\rank}{\operatorname{rank}}
\newcommand{\R}{\mathbb{R}}
\newcommand{\N}{\mathbb{N}}
\newcommand{\D}{\mathcal{D}}
\newcommand{\T}{T}
\newcommand{\M}{\mathrm{M}}
\newcommand{\Span}{\mathrm{Span}}
\newcommand{\Hom}{\mathrm{Hom}}
\newcommand{\V}{\mathcal{V}}
\newcommand{\W}{\mathcal{W}}

\newcommand{\p}{\mathfrak{p}}

\newcommand{\gl}{\mathfrak{gl}}

\newcounter{theoremintro}

\newcounter{countfactD}

\newcounter{countfactDD}

\newcounter{countquestionD}

\newcounter{countquestionDD}

\theoremstyle{plain}
\newtheorem{question}{Question}
\newtheorem{questionD}[countquestionD]{Question}

\theoremstyle{plain}
\newtheorem{fact}{Fact}
\newtheorem{factD}[countfactD]{Fact}
\newtheorem{factDD}[countfactDD]{Fact}

\newtheorem{Theorem*}[theoremintro]{Theorem}

\theoremstyle{plain}
\newtheorem{theoremx}{Theorem}

\theoremstyle{plain}
\newtheorem{thmxx}{Theorem}

\newtheorem{corollaryx}[theoremx]{Corollary}
\newtheorem{corollaryxx}{Corollary}

\makeatletter
\def\Ddots{\mathinner{\mkern1mu\raise\p@
		\vbox{\kern7\p@\hbox{.}}\mkern2mu
		\raise4\p@\hbox{.}\mkern2mu\raise7\p@\hbox{.}\mkern1mu}}
\makeatother

\begin{document}
\thispagestyle{empty}
% \renewcommand{\PaperNumber}{***}

% \FirstPageHeading

\ShortArticleName{$B'$-orbits on flag varieties and symmetry breaking}

\ArticleName{$B'$-orbits on flag varieties and symmetry breaking}

% % Names of the authors for the title of the paper
\Author{Valentin MASSICOT}

\AuthorNameForHeading{V.~Massicot}

\Address{Université de Reims Champagne-Ardenne, UMR 9008, Laboratoire de Mathématiques de Reims, 51100 Reims, France}
\Address{French-Japanese Laboratory of Mathematics and its Interactions, IRL2025 CNRS and The University of Tokyo, Tokyo, Japan}
%  % Address of First Author
% \EmailD{\mail{valentin.massicot@univ-reims.fr}} % E-mail address of First Author
% \URLaddressD{\url{https://valentinmassicot.perso.math.cnrs.fr/}} %URL address of First Author

% % In the case of the same organization, please use the following standard
% %\Author{First Names LASTNAME and Second COAUTHOR}
% %\AuthorNameForHeading{F.N. Lastname and S. Coauthor}
% %\Address{Address of Author(s), Country}
% %\Email{\mail{email1@address}, \mail{email2@address}}
% %\URLaddress{\url{http://www.home.org/~myHome1/}}

% \ArticleDates{Received ???, in final form ????; Published online ????}

\Abstract{Motivated by branching problems for principal series representations of the Lie group $G = GL(n,\R)$, we consider all pairs $(G', P)$ with $G'$ being the Levy factor of a parabolic subgroup of $G$ and $P$ a parabolic subgroup of $G$ for which a Borel subgroup $B'$ of $G'$ has finitely many orbits on $G/P$.
We classify all such pairs $(G',P)$ for which $B'$-orbits on the generalized flag variety $G/P$ are determined by invariant functions inspired from the Bruhat decomposition.
We also describe explicitly the double coset space $B'\backslash G/P$ as well as the closed $B'$-orbits on $G/P$ whenever $B'$-orbits are computed by these invariant functions.}

\Keywords{double coset space; flag variety; Bruhat decomposition; invariant functions; branching problems; symmetry breaking} %up to 6 key words
% %Please type here List of Keywords for your article separated by semicolon.

\Classification{22E45; 57S25; 14M15} % e.g. 35A30; 81Q05
% % For 2020 Mathematics Subject Classification see https://mathscinet.ams.org/mathscinet/msc/msc2020.htmlq

\section{Introduction}

Let $G' \subset G$ be a pair of groups and $(V,\pi)$ be an irreducible representation of $G$. Branching problems ask about the behaviour of the restriction $\pi_{|G'}$. If $\pi$ happens to be unitary and $G$ and $G'$ locally compact, the Mautner theorem (see \cite[Theorem 1.2.]{Mautner}) implies that $\pi_{|G'}$ decomposes (uniquely if $G$ is a type I group) as a direct integral of irreducible unitary representations of $G'$. When $\pi$ is no longer unitary (for instance a smooth representation of a Lie group), there is no universal notion of irreducible decomposition. In this broader context, one may study restrictions of representations through the space of symmetry breaking operators (SBOs for short) which are elements of $\Hom_{G'}(\pi_{|G'},\tau)$ where $\tau$ ranges over an appropriate class of representations of $G'$. 

We are interested in the setting where $G' \subset G$ is a pair of real reductive groups. In this configuration, a theorem of Casselman (see \cite[Theorem 8.37]{Knapp}) asserts that every irreducible admissible representation of $G$ may be embedded in some principal series representation. Moreover, these principal series representations are known to be generically irreducible (see \cite[Theorem 7.2]{bruhat} for unitary principal series) so they will constitute our "appropriate class of representations". These representations of $G$ are obtained by induction with respect to parabolic subgroups and can thus be realized on sections of vector bundles over some generalized flag varieties. This geometric realization allows us to use functional analysis techniques to tackle down branching problems. More concretely, if $P$ is a parabolic subgroup of $G$ and $\chi$ is a finite dimensional representation of $P$, the quotient space $\mathcal V$ of $G \times V$ by the equivalence relation $(g,v) \sim (gp,\chi^{-1}(p)v)$ is a homogeneous vector bundle over $G/P$ and the induced representation $\operatorname{ind}_P^G \chi$ may be realized on the space of smooth sections $C^\infty(G/P,\mathcal V)$. If in addition we consider a parabolic subgroup $P'$ of $G'$ and an homogeneous vector bundle $\mathcal W$ over $G'/P'$, the Schwartz kernel theorem (\cite[Proposition 25]{schwartz}) implies that $G'$-intertwining operators $T : C^\infty(G/P,\V) \to C^\infty(G'/P',\W)$ correspond to $P'$-equivariant distribution kernels $K_T \in (D'(G/P,\V^*) \otimes W)^{P'}$ where $W$ is the fiber of $\W$ (see \cite[Proposition 3.2]{KobayashiSpeh}). Note that the support of such a distribution kernel is a $P'$-invariant closed subset of $G/P$. \\
This fact leads us to the following strategy: the topological analysis of $P'$-orbits of $G/P$ provides an information on $K_T$, hence on $T$. More concretely, one may translate the equivariance of $K_T$ as a system of PDEs and use the knowledge of the support of $K_T$ as "boundary conditions" in order to solve this system (see \cite[Chapter 3]{KobayashiSpeh2}). \\
One way to describe the topological structure of $P'\backslash G/P$ is through the partial order defined as
\begin{equation*}
	\forall X_1,X_2 \in P'\backslash G / P, \hspace{3mm} X_1 \prec X_2 \overset{\mathrm{def}}{\iff} X_1 \subset \overline{X_2}
\end{equation*}
and the corresponding Hasse diagram (see \Cref{GL3-GL2} and \Cref{hasse_diagram2}).

Over the last decade, the properties of such double coset spaces have been studied extensively in the context of representation theory in relation to

\begin{itemize}[leftmargin=*, label = --, noitemsep, nolistsep]
	% \item If $P'$ is minimal and has an open orbit on $G/P$, $\Hom_{G'}(\pi_{|G'},\tau)$ is finite-dimensional for all irreducible admissible representations $\pi$, $\tau$ of $G$ and $G'$ (see \cite{KobayashiOshima} for the precise statement). Moreover, its dimension is the number of open $P'$-orbits for generic principal series representations $\pi$ and $\tau$. In this setting, each open orbit can be used to obtain a family of SBOs depending holomorphicaly on the induction parameters (see \cite{KobayashiSpeh}, \cite{KobayashiSpeh2} for the general scheme and its application to the symmetric pair $(G,G') = (O(n,1),O(n-1,1))$),
	
	\item construction of symmetry breaking operators for branching problems of principal series representations. In this context, for $P'$ minimal which has an open orbit on $G/P$, $\Hom_{G'}(\pi_{|G'},\tau)$ is finite-dimensional for all irreducible admissible representations $\pi$, $\tau$ of $G$ and $G'$ (see \cite{KobayashiOshima} for the precise statement). Moreover, its dimension is the number of open $P'$-orbits for generic principal series representations $\pi$ and $\tau$. In this setting, each open orbit can be used to obtain a family of SBOs depending meromorphicaly on the induction parameters (see \cite{KobayashiSpeh}, \cite{KobayashiSpeh2} for the general scheme and its application to the symmetric pair $(G,G') = (O(n,1),O(n-1,1))$, see \cite{Leontiev} for the symmetric pair $(G,G') = (O(p,q),O(p-1,q))$),
	
	\item analysis of symmetry breaking operators. More precisely, starting from meromorphic families of intertwining operators supported on an open orbit, one may use residue calculus to obtain SBOs supported on smaller orbits (see \cite{KobayashiSpeh}). Furthermore, T. Kobayashi and M. Pevzner proposed a new method based on a so-called generating operator to obtain SBOs supported on large orbits, starting from known SBOs supported on small orbits (see \cite{normal_derivatives}, \cite{Generating_Rankin_Cohen}, \cite{discretetocontinuous}). These procedures can be thought as going up or down in the Hasse diagram of $P'\backslash G/P$.
	
	\item discrete decomposability in the framework of the BGG category $\mathcal O^\p$ (see \cite{KobayashiVerma}).
	
	% \item A symmetry breaking operator $T : C^\infty(G/P,\V) \to C^\infty(G'/P',\W)$ is said to be differential with respect to an embedding $p : G'/P' \to G/P$ if it does not increase the support, that is $\supp \, T(f) \subset p(\supp f)$. One may prove that a SBO $T$ is differential in this sense if and only if $\supp \, K_T$ is a fixed point of the $P'$-action on $G/P$ (see for instance \cite[Lemma $2.22$]{fmethod}). This fact can then be used to classify differential SBOs in concrete settings (see for instance \cite{DitlevsenLabriet} for the symmetric pair $(GL(n,\R),GL(n-1,\R))$ which is related to our study).
	\item differential symmetry breaking operators. In fact, differential symmetry breaking operators are caracterized among all symmetry breaking operators by the fact that their distributional kernel is supported at a $P'$-fixed point of $G/P$ (see for instance \cite[Lemma $2.22$]{fmethod}). This fact can then be used to classify differential SBOs in concrete settings (see for instance \cite{DitlevsenLabriet} for the symmetric pair $(GL(n,\R),GL(n-1,\R))$ which is related to our study).
\end{itemize}

\medskip

In this work, we are interested in the setting where $G = GL(n,\R)$, $G'= GL(n_1,\R) \times \cdots \times GL(n_k,\R)$ is the Levi factor of a parabolic subgroup of $G$, $P$ is a parabolic subgroup of $G$ and $P'= B'$ the standard Borel subgroup of $G'$, namely the set of upper triangular matrices in $G'$.

An extreme case is given by $G' = G$ and $P = B$ the Borel subgroup of $G$. In this situation, the Bruhat decomposition asserts that $B \backslash G/B$ is in bijection with the symmetric group $\mathfrak S_n$ under the map $\sigma \mapsto BM_\sigma B$ where $M_\sigma$ is the permutation matrix associated to $\sigma$.
Moreover, it is known (see \cite[§10.5]{fulton_tableaux}) that both the $B$-orbits on $G/B$ and the partial order on $B\backslash G/B$ are determined by the maps
\begin{equation*}
	r_{ij} : \begin{array}{ccc}
	\mathfrak S_n & \to & \N_{\geqslant 0} \\
	\sigma & \mapsto & \operatorname{Card}\{s\in \{1,\dots,j\} \mid \sigma(s) \geqslant n-i+1\}
\end{array}
\end{equation*}
in the following sense:
\begin{fact}[Combinatorial formulation]\label{Fact Bruhat}
	The map 
	\begin{equation*}
		r : \begin{array}{ccc}
	B\backslash G/B & \to & \N_{\geqslant 0}^{\frac{(n-2)(n-1)}{2}} \\
	B M_\sigma B & \mapsto & (r_{ij}(\sigma))_{1 \leqslant i,j \leqslant n-1}
\end{array}
	\end{equation*}
	is injective and induces a poset isomorphism onto its image.
\end{fact}
Note that for a permutation $\sigma \in \mathfrak S_n$, one has $r_{ij}(\sigma) = \rank(M_\sigma^{ij})$ where given a matrix $g$, $g^{ij}$ is the submatrix obtained from $g$ by intersecting the last $i$ rows with the first $j$ columns.
Since the rank of such submatrices is a $(B \times B)$-invariant (i.e. invariant under both right and left translation by $B$), we may rephrase \Cref{Fact Bruhat} as
\begin{factD}[Invariant-theoretic formulation]
	The map 
	\begin{equation*}
		r : \begin{array}{ccc}
	B\backslash G/B & \to & \N_{\geqslant 0}^{\frac{(n-2)(n-1)}{2}} \\
	BgB & \mapsto & (\rank(g^{ij}))_{1 \leqslant i,j \leqslant n-1}
\end{array}
	\end{equation*} is a well-defined injection and induces a poset isomorphism onto its image.

\end{factD}
This fact may be formulated more algebraicly. Indeed, if we denote by $P_i$ the subgroup
\begin{equation*}
	P_i \coloneq \left\{\begin{pmatrix}
x & y \\
0 & z
\end{pmatrix}x \in GL(i,\R), z \in GL(n-i,\R), y \in M_{i,n-i}(\R)\right\} \subset GL(n,\R),
\end{equation*}
then $(P_i)_{1 \leqslant i \leqslant n-1}$ is exactly the set of all (both standard and non-standard) maximal parabolic subgroups of $G$ containing $B$ and the map
\begin{equation*}
	\begin{array}{ccc}
	GL(n,\R) & \to & \N_{\geqslant 0}^{\frac{(n-2)(n-1)}{2}} \\
	g & \mapsto & (\rank(g^{ij}))_{1 \leqslant i,j \leqslant n-1}
\end{array}
\end{equation*}
induces a bijection between $P_i\backslash G/P_j$ and its image. Consequently, we may state \Cref{Fact Bruhat} as 
\begin{factDD}[Algebraic formulation]
	The map 
	\begin{equation*}
		\begin{array}{ccc}
	B\backslash G/B & \to & \displaystyle \prod_{1 \leqslant i,j \leqslant n-1}P_i \backslash G/P_j \\
	BgB & \mapsto & (P_igP_j)_{1 \leqslant i,j \leqslant n-1}
\end{array}
	\end{equation*}
	is injective and induces a poset isomorphism onto its image.
\end{factDD}

Let us return to the case where $G'= GL(n_1,\R) \times \cdots \times GL(n_k,\R)$ and $P$ is a general parabolic subgroup of $G$. We ask the following question:
\setcounter{question}{1}
\begin{question}[Invariant-theoretic formulation]\label{Question 2}
	When are $B'$-orbits on $G/P$ determined by the $(B' \times P)$-invariant maps $g \mapsto \rank(g^{ij})$?
\end{question}
Again, we may formulate it more algebraically using the sets $\mathcal Q_{B'}, \mathcal Q_{P}$ of (possibly non-standard) maximal parabolic subgroups of $G$ containing $B'$ and $P$ respectively.
\setcounter{countquestionD}{1}
\begin{questionD}[Algebraic formulation]
	When is the map 
	\begin{equation*}
		\Psi : \begin{array}{ccc}
		B' \backslash G/P & \to & \displaystyle \prod_{Q_1 \in \mathcal Q_{B'}, Q_2 \in \mathcal Q_{P}} Q_1 \backslash G/Q_2 \\
		B'gP & \mapsto & (Q_1gQ_2)_{\substack{Q_1 \in \mathcal Q_{B'} \\ Q_2 \in \mathcal Q_{P}}}
	\end{array}
	\end{equation*}
	injective?
\end{questionD}

Considering there is only finitely many possible choices of a submatrix, hence of rank functions (and finitely many values for each rank), the codomain of $\Psi$ is finite.
In particular, $\Psi$ can only be injective if $B'\backslash G/P$ is finite.
The classification of the pairs $(G',P)$ for which $B'$ has finitely many orbits on $G/P$ has been achieved by T. Kobayashi using visible actions on complex manifolds in \cite{ActionVisible} and will be our starting point (see \Cref{flag varieties} for the definition of $P_{\underline m}$).

\setcounter{fact}{2}
\begin{fact}[\cite{ActionVisible}]\label{Fact Kob}
	The pairs $(G',P)$ for which $B'$ has finitely many orbits on $G/P$ are precisely the pairs $(GL(n_1,\R)\times \cdots \times GL(n_k,\R), P_{\underline m})$ satisfying one of the following conditions:
\begin{equation*}
	\begin{array}{rcccc}
		0) & k = 2, & \text{any } N, & l=2, & \text{any } M,\\
		\mathrm{I}) & k = 3, & N = 1, & l = 2, & M \geqslant 2,\\
		\mathrm{II}) & k = 3, & N \geqslant 2, & l = 2, & M = 2, \\
		\mathrm{III}) & \text{any } k, & \text{any } N, & l = 2, & M = 1, \\
		\mathrm{I}') & k = 2, & N \geqslant 2, & l = 3, & M = 1,\\
		\mathrm{II}') & k = 2, & N = 2, & l = 3, & M \geqslant 2,\\
		\mathrm{III}') & k = 2, & N = 1, & \text{any } l, & \text{any } M, \\
	\end{array}
\end{equation*}
where $N = \min(\underline n)$ and $M = \min(\underline m)$.
\end{fact}

Our main result is the following theorem (see \Cref{Theorem A} for the precise statement and \Cref{def Q_H} for notations).

\begin{Theorem*}
	Assume that $B'\backslash G/P$ is finite.
	\begin{enumerate}[label = \roman*)]
		\item If condition $\mathrm{I}')$ of \Cref{Fact Kob} holds with $m_1 = 1$, $m_2 \neq 1$ and $n \geqslant 5$ or if condition $\mathrm{II}')$ of \Cref{Fact Kob} holds, then the map $\Psi$ is not injective.
		\item In all remaining cases, $\Psi$ is injective.
	\end{enumerate}
\end{Theorem*}
As explained previously, the injectivity of the map $\Psi$ means that the $B'$-orbits on $G/P$ are determined by the $(B' \times P)$-invariant ranks introduced in \Cref{def rang}.
This gives a rather explicit way to compute the $B'$-orbit of any element in $G/P$ as well as a theoretical reason why these double quotients are finite: aside from $\mathrm{II}')$ and a specific subcase of $\mathrm{I}')$, they are exactly the ones for which our rank functions entirely determine the $B'$-orbits.
In particular, this result is purely algebraic and holds not only on $\R$ but for any local field.
We also obtained an explicit description of the $B'$-orbits on $G/P$ in the case where $\Psi$ is injective. Finally, we determined the closed $B'$-orbits on $G/P$ and found that they are exactly the $B'$-fixed points on $G/P$.

% We plan on using these explicit description of $B'$-orbits to compute the residues of the families of SBOs introduced in \cite{DitlevsenLabriet}, thereby pursuing the study of restriction from $GL(n,\R)$ to $GL(n-1,\R)$ initiated in \cite{Frahm_2023}, \cite{DitlevsenFrahm} and \cite{DitlevsenLabriet}.
% One could also deduce from our work the $P'$-orbits on $G/P$ for any parabolic subgroup $P'$ of $G'$ by considering the equality $P' = LB'$ where $L$ is the Levi factor of $P'$.

\section{Preliminaries}

In this section, we recall the notion of flag variety and introduce some coordinate systems in order to do explicit computations with flags. We also introduce the invariant maps involved in \Cref{Question 2} and \Cref{Theorem A}.

\subsection{Parabolic subgroups and flag varieties}{\label{flag varieties}}

Let $n \in \N_{>0}$ be a positive integer. A composition of $n$ is a finite sequence $\underline n = (n_1,\dots,n_k)$ of positive integers satisfying $n_1+\cdots+n_k = n$ with $k \in \N_{>0}$. Another composition $(m_1,\dots,m_l)$ of length $l \in \N_{>0}$ of $n$ will be called a subcomposition of $\underline n$ if there exists a composition $(i_1,\dots,i_l)$  of $k$ such that 
\begin{equation}\label{eqn:subcomposition}
	\forall j \in \{1,\dots,l\}, \hspace{3mm}m_j = \sum_{s = 1}^{i_j} n_{i_1+\cdots+i_{j-1} + s}.
\end{equation}

If $\underline n = (n_1,\dots,n_k)$ is a composition of $n$, define $P_{\underline n}$ to be the standard parabolic subgroup of $GL(n,\R)$ given by

\begin{center}
	\begin{tikzpicture}[scale = 0.5,
		%Global config
		>=latex,
		line width=1pt,
		%Styles
		Brace/.style={
			decorate,
			decoration={
				brace,
				raise=-7pt
			}
		}
		]
		
		\matrix[% General option for all nodes
		matrix of nodes,
		text height=1.25ex,
		text depth=0.375ex,
		text width=1.625ex,
		align=center,
		left delimiter=(,
		right delimiter= ),
		column sep=2.5pt,
		row sep=2.5pt,
		%nodes={draw=black!10}, % Uncoment to see the square nodes.
		nodes in empty cells,
		] at (0,0) (M){ % Matrix contents  
			& & & & & & \\
			& & & & & & \\
			& & & & & & \\
			& & & & & & \\
			& & & & & & \\
			& & & & & & \\
			& & & & & & \\
		};
		% Drawing the sectors using matrix coordinate names.   
		\draw[thick,draw] (M-3-3.south east)
		-- (M-3-2.south)
		-- (M-2-2.center)
		-- (M-2-1.west)
		-- (M-1-1.north west)
		-- (M-1-7.north east)
		-- (M-7-7.south east)
		-- (M-7-6.south)
		-- (M-6-6.center)
		-- (M-6-5.west)
		-- (M-5-5.north west);
		% Drawing the braces.   
		\draw[Brace] (M-1-1.north)
		-- (M-2-2.north)
		node[midway]{\small $n_1$};
		
		\draw[Brace] (M-2-2.east)
		-- (M-3-3.east)
		node[midway]{\small $n_2$};
		
		\draw[Brace] (M-6-6.east)
		-- (M-7-7.east)
		node[midway]{\small $n_k$};
		
		\node at (0,0) {\rotatebox{-15}{$\ddots$}};
		
		\node at (-6.2,0) {$P_{\underline n} =$};
		
		\node at (5.6,0) {.};
	\end{tikzpicture}
\end{center}

In this work, the term "parabolic subgroup" will refer to any subgroup of $GL(n,\R)$ which is conjugate to a standard parabolic subgroup by a permutation matrix. These are exactly the parabolic subgroups corresponding to the standard Cartan subalgebra of $\gl(n,\R)$ equipped with a non-standard basis of its root system.

Recall (see for instance \cite[Section $1.2$]{Brion}) that given a composition $\underline n = (n_1,\dots,n_k)$ of $n$, a flag of type $\underline n$ is a sequence
\begin{equation*}
	V_1 \subset \cdots \subset V_{k-1} \subset V_k = \R^n
\end{equation*}
of subspaces satisfying $\dim V_i = n_1+\cdots+n_i$. The group $GL(n,\R)$ acts transitively on the set $\D_{\underline n}$ of flags of type $\underline n = (n_1,\dots,n_k)$. Moreover, if $(e_1,\dots,e_n)$ is the standard basis of $\R^n$, the isotropy subgroup of the flag
\begin{equation*}
	\Span(e_1,\dots,e_{n_1}) \subset \Span(e_1,\dots,e_{n_1+n_2}) \subset \cdots \subset \Span(e_1,\dots,e_{n_1+\cdots+n_k}) = \R^n
\end{equation*}
is the standard parabolic subgroup $P_{\underline n}$. This induces a bijection $\mathcal \D_{\underline n} \simeq GL(n,\R)/P_{\underline n}$ from which we define a topology on $\mathcal \D_{\underline n}$, turning the bijection into a homeomorphism.

If $A \in GL(n,\R)$ is decomposed as
\begin{equation*}
	A = \begin{pmatrix}
	A_1 & \cdots & A_{k}
\end{pmatrix}
\end{equation*}
with each $A_i \in M_{n,n_i}(\R)$, we will write the corresponding flag $D = (V_0,\dots,V_k) \in \D_{\underline n} \simeq G/P_{\underline n}$ as
\begin{equation*}
	D = \left[\begin{tabular}{c|c|c}
	$A_1$ & $\cdots$ & $A_{k}$
\end{tabular}\right].
\end{equation*}
Since the block $A_k$ corresponds to the subspace $V_k = \R^n$, changing it has no effect on $D$ and $D$ is entirely determined by $A_1,\dots,A_{k-1}$. The invertibility of $A$ implies that
\begin{equation*}
	\begin{pmatrix}
	A_1 & \cdots & A_{k-1}
\end{pmatrix}
\end{equation*}
has maximal rank. In particular, we can also realize the flag variety $\D_{\underline n}$ as the quotient of the space of maximal rank matrices $M_{n,n-n_k}(\R)_{\mathrm{reg}}$ under the right action of $P_{(n_1,\dots,n_{k-1})} \subset GL(n_1+\cdots+n_{k-1})$.

%If $A \in M_{n,n-n_k}(\R)_{\mathrm{reg}}$ is decomposed as
%$$A = \begin{pmatrix}
%	A_1 & \cdots & A_{k-1}
%\end{pmatrix}$$
%with each $A_i \in M_{n,n_i}(\R)$, we will also write the corresponding flag $D \in \D_{\underline n}$ as
%$$D = \left[\begin{tabular}{c|c|c}
%	$A_1$ & $\cdots$ & $A_{k-1}$
%\end{tabular}\right].$$
%Clearly, one may obtain the $M_{n,n-n_k}(\R)_{\mathrm{reg}}/P_{(n_1,\dots,n_{k-1})}$-description of a given flag from its $G/P_{\underline n}$-description by removing the block $A_k$ (in particular, changing the block $A_k$ has no effect on $D$).

The third description of $\D_{\underline n}$ (as a quotient of $M_{n,n-n_k}(\R)_{\mathrm{reg}}$)is more convenient for explicit computations as it contains less redundancy but the second one (as a quotient of $GL(n,\R)$) may allow us to describe specific flags via permutation matrices (see \Cref{Theorem B} for instance).

\begin{Remark}
	Since we realize flags as equivalence classes of invertible (resp. maximal rank) matrices, there is no unique way of representing a specific flag. More precisely, the action of the Levi part of $P_{\underline n}$ (resp. $P_{(n_1,\dots,n_{k-1})}$) allows us to perform arbitrary changes of basis in each block $A_i$ and its unipotent radical allows us to add any linear combination of columns of a block $A_i$ to any column of a different block $A_j$ as long as $i < j$.
\end{Remark}

\begin{Definition}{\label{def projections}}
	Consider for $n \in \N_{>0}$ a composition $\underline n = (n_1,\dots,n_k)$ of $n$ and a subcomposition $(m_1,\dots, m_l)$ of $\underline n$, such that $(i_1,\dots,i_l)$ satisfies \eqref{eqn:subcomposition}. We define the natural projection
	\begin{equation*}
		\Phi_{\underline n \to \underline m} : \begin{array}{ccc}
		\D_{\underline n} & \to & \D_{\underline m} \\
		(V_1 \subset \cdots \subset V_{k-1}) & \mapsto & \displaystyle \Big(V_{i_1} \subset V_{i_1+i_2} \subset \cdots \subset V_{i_1+\cdots+i_{l-1}}\Big)
	\end{array}.
	\end{equation*}
	In the special case where $\underline m = (n_1+\cdots + n_s,n-n_1-\cdots-n_s)$ for some $s \in \{1,\dots,k\}$, we will write $D_s$ for the projection of $D \in \D_{\underline n}$ onto $\D_{\underline m}$.
\end{Definition}

\begin{Remark}
	In our matrix realizations of $\D_{\underline n}$, the projection $\D_{\underline n} \to \D_{\underline m}$ may be realized by removing the last $m_l-1$ blocks describing a flag and merging the first $m_1$ blocks together, the next $m_2$ blocks together etc.
	
	Moreover, if $\D_{\underline n}$ and $\D_{\underline m}$ are realized as quotients of $GL(n,\R)$, the map $\Phi_{\underline n \to \underline m}$ is exactly the canonical map $GL(n,\R)/P_{\underline n} \to GL(n,\R)/P_{\underline m}$.
\end{Remark}

\begin{Example}{\label{exemple flags}}
	Let us consider $n = 4$, $\underline n = (1,1,1,1)$, $\underline m = (1,2,1)$ and let $(e_1,e_2,e_3,e_4)$ be the canonical basis of $\R^4$. The complete flag
	\begin{equation*}
	D = \Big(\Span(e_1+2e_2) \subset \Span(e_1+2e_2,e_1+e_3) \subset \Span(e_1+2e_2,e_1+e_3,e_4) \Big)
	\end{equation*}
	can be represented (using our second realization of $\D_{\underline n})$ as
	\begin{equation*}
		D = \left[\begin{tabular}{c|c|c}
		$1$ & $1$ & $0$ \\
		$2$ & $0$ & $0$ \\
		$0$ & $1$ & $1$ \\
		$0$ & $0$ & $0$
	\end{tabular}\right] = \left[\begin{tabular}{c|c|c}
		$1$ & $0$ & $0$ \\
		$2$ & $-2$ & $0$ \\
		$0$ & $1$ & $1$ \\
		$0$ & $0$ & $0$
	\end{tabular}\right] = \left[\begin{tabular}{c|c|c}
		$1$ & $0$ & $0$ \\
		$2$ & $1$ & $0$ \\
		$0$ & $-\frac{1}{2}$ & $1$ \\
		$0$ & $0$ & $0$
	\end{tabular}\right]
	\end{equation*}
	and its projection onto $\D_{(1,2,1)}$ is
	\begin{equation*}
		\Phi_{\underline n \to \underline m}(D) = \left[\begin{tabular}{c|cc}
		$1$ & $1$ & $0$ \\
		$2$ & $0$ & $0$ \\
		$0$ & $1$ & $1$ \\
		$0$ & $0$ & $0$
	\end{tabular}\right] = \left[\begin{tabular}{c|cc}
		$1$ & $0$ & $0$ \\
		$0$ & $1$ & $0$ \\
		$1$ & $-\frac{1}{2}$ & $1$ \\
		$0$ & $0$ & $0$
	\end{tabular}\right] = \left[\begin{tabular}{c|cc}
	$1$ & $0$ & $0$ \\
	$0$ & $1$ & $0$ \\
	$1$ & $0$ & $1$ \\
	$0$ & $0$ & $0$
	\end{tabular}\right].
	\end{equation*}
\end{Example}

\begin{Lemma}{\label{projection equivariant}}
	Let $\underline n$ be a composition of $n$ and $\underline m$ be a subcomposition of $\underline n$. Then, the map $\Phi_{\underline n \to \underline m}$ is $GL(n,\R)$-equivariant.
\end{Lemma}
\begin{proof}
	Consider the identifications $\mathcal \D_{\underline n} \simeq GL(n,\R)/P_{\underline n}$ and $\mathcal \D_{\underline m} \simeq GL(n,\R)/P_{\underline m}$. If $g,A \in GL(n,\R)$, one has
	\begin{equation*}
		\Phi_{\underline n \to \underline m}(g \cdot AP_{\underline n})  = \Phi_{\underline n \to \underline m}(gAP_{\underline n}) = gAP_{\underline m} = g \cdot AP_{\underline m} = g \cdot \Phi_{\underline n \to \underline m}(AP_{\underline n}).
	\end{equation*}
\end{proof}

\subsection{Ranks of flags}

\begin{Definition}{\label{def rang}}
	Let $\underline n = (n_1,\dots,n_k)$ be a composition of $n$, $J \subset \{1,\dots,n\}$, $s \in \{1,\dots,k-1\}$ and $\D_{\underline n}$ be the flag variety associated to $\underline n$ (see \Cref{flag varieties}). The map $\rank_{J,s} : \D_{\underline n} \to \N_{\geqslant 0}$ is defined by
	\begin{equation*}
		\forall D \in \D_{\underline n}, \hspace{3mm} \rank_{J,s}(D) = \rank\Big((a_{ij})_{(i,j) \in J \times \{1,\dots,n_1+\cdots+n_s\}}\Big)
	\end{equation*}
	where $(a_{ij}) \in M_{n,n_1+\cdots+n_s}(\R)_{\mathrm{reg}}$ is a representative for $D_s \in \D_{(n_1+\cdots+n_s,n-n_1+\cdots+n_s)}$.
\end{Definition}

\begin{Remark}
	Note that these maps are well defined since $\rank_{J,s} (D)$ coincides with $\dim p_J(D_s)$ where $p_J : \R^n \to \R^{n}$ is the projection onto $\Span(e_j, j \in J) \subset \R^n$.
\end{Remark}

\begin{Remark}
	These maps are lower semi-continuous because the usual rank of matrices is.
	In particular, if $X \subset \D_{\underline n}$ is a set of matrices satisfying $\rank_{J,s} D = r$ for all $D \in X$, then any $D \in \overline X$ must satisfy $\rank_{J,s} D \leqslant r$.
\end{Remark}

\begin{Remark}
	If $k = 2$, we shall write $\rank_{J}$ instead of $\rank_{J,1}$.
\end{Remark}

\begin{Example}
	Let $(e_1,e_2,e_3,e_4)$ be the canonical basis of $\R^4$. Recall from \Cref{exemple flags} that the complete flag
	\begin{equation*}
		D = \Big(\Span(e_1+2e_2) \subset \Span(e_1+2e_2,e_1+e_3) \subset \Span(e_1+2e_2,e_1+e_3,e_4) \Big)
	\end{equation*}
	can be represented as $D = \left[\begin{tabular}{c|c|c}
		$1$ & $1$ & $0$ \\
		$2$ & $0$ & $0$ \\
		$0$ & $1$ & $0$ \\
		$0$ & $0$ & $1$
	\end{tabular}\right]$. Since $D_2 = \begin{bmatrix}
		1 & 1 \\
		2 & 0 \\
		0 & 1 \\
		0 & 0
	\end{bmatrix}$, one has
	\begin{equation*}
		\rank_{\{1,4\},2}(D) = \rank\begin{pmatrix}
		1 & 1 \\
		0 & 0
	\end{pmatrix} = 1, \hspace{1cm} \rank_{\{2,3\},2}(D) = \rank\begin{pmatrix}
		2 & 0 \\
		0 & 1
	\end{pmatrix} = 2.
	\end{equation*}
\end{Example}

\subsection[Invariant ranks]{$(B' \times P_{\underline m})$-invariant ranks}{\label{subsection invariant}}

Let $\underline n = (n_1,\dots,n_k)$ be a composition of $n$, $J \subset \{1,\dots,n\}$, $s \in \{1,\dots,k-1\}$ and consider the map $\rank_{J,s} : \D_{\underline n} \to \N_{\geqslant 0}$. The subgroup
\begin{equation*}
	H = \{g \in GL(n,\R) \mid \forall D \in \D_{\underline n}, \, \rank_{J,s}(g \cdot D) = \rank_{J,s}(D)\}
\end{equation*}
admits then an alternative description:
\begin{equation*}
	H = \{(a_{ij}) \in GL(n,\R) \mid \forall i \notin J, \, \forall j \in J, \, a_{ij} = 0\}.
\end{equation*}
Note that the standard parabolic subgroup $P_{(n-|J|,|J|)}$ is conjugate to $H$: 
\begin{equation*}
	P_{(n-|J|,|J|)} = M_{\sigma}^{-1}HM_\sigma
\end{equation*}
where $M_\sigma$ is the matrix of any permutation $\sigma$ satisfying
\begin{equation*}
	\sigma(J) = \{n-|J|+1,n-|J|+2,\dots,n\}.
\end{equation*}
In particular, $H$ is a (non-standard) maximal parabolic subgroup of $G$ which only depends on $J$. We will denote it by $P_J$.

\begin{Lemma}{\label{lemme_rank}}
	Let $\underline m = (m_1,m_2)$ be a composition of $n$. The $P_J$-orbits on $\D_{\underline m}$ are parametrized by $\rank_{J} \equiv \rank_{J,1}$.
\end{Lemma}
\begin{proof}
	It suffices to prove the statement for the standard parabolic subgroup $P_{(|J|,n-|J|)}$. To this end, we let $D \in \D_{\underline m}$ be given as
	\begin{equation*}
		D = \begin{bmatrix}
		A_1 & \cdots & A_{m_1} \\
		B_1 & \cdots & B_{m_1}
	\end{bmatrix}
	\end{equation*}
	with $A_1,\dots,A_{m_1} \in \R^{n-|J|}$ and $B_1,\dots,B_{m_1} \in \R^{|J|}$. Using the lower right block of the Levi part of $P_{(|J|,n-|J|)}$ and a change of basis allowed by the right action of $GL(m_1,\R)$ on $M_{n,m_1}(\R)_{\mathrm{reg}}$, we find that $D$ is in the same orbit as the element
	\begin{equation*}
		\begin{bmatrix}
		A_1' & \cdots & A_r' & A_{r+1}' & \cdots & A_{m_1}' \\
		f_1 & \cdots & f_r & 0 & \cdots & 0
	\end{bmatrix}
	\end{equation*}
	where $r = \rank_{\{n-|J|+1,\dots,n\}}D$ and $A_1',\cdots, A_{m_1}' \in \R^{n-|J|}$. Using the action of the unipotent radical of $P_{(|J|,n-|J|)}$, we then find that $D$ is in the same orbit as the element
	\begin{equation*}
		\begin{bmatrix}
		0 & \cdots & 0 & A_{r+1}' & \cdots & A_{m_1}' \\
		f_1 & \cdots & f_r & 0 & \cdots & 0
	\end{bmatrix}.
	\end{equation*}
	Since the above matrix has rank $m_1$, the family $(A_{r+1}',\dots,A_{m_1}')$ is linearly independent and we use the upper left block of the Levi part of $P_{(|J|,n-|J|)}$ to find that
	$D$ belong to the same orbit as the element
	\begin{equation*}
		\begin{bmatrix}
		0 & \cdots & 0 & e_1 & \cdots & e_{m_1-r} \\
		f_1 & \cdots & f_r & 0 & \cdots & 0
	\end{bmatrix}.
	\end{equation*}
	Since $\rank_{\{n-|J|+1,\dots,n\}}$ is $P_{(|J|,n-|J|)}$-invariant, this proves the lemma.
\end{proof}

\section{Main results}

Let $\underline n, \underline m$ be compositions of $n \in \N_{>0}$, $G = GL(n,\R)$, $G' = GL(n_1,\R) \times \cdots \times GL(n_k,\R) \subset G$, $P_{\underline m}$ be a standard parabolic subgroup of $G$ and $B'$ be the standard Borel subgroup of $G'$.

According to \cite[Theorem A]{ActionVisible}, the pairs $(G',P)$ for which $B'$ has finitely many orbits on $G/P$ are precisely the pairs $(GL(n_1,\R)\times \cdots \times GL(n_k,\R), P_{\underline m})$ satisfying one of the following conditions:
\begin{equation}\label{eqn:kobcon}
	\begin{array}{rcccc}
		0) & k = 2, & \text{any } N, & l=2, & \text{any } M,\\
		\mathrm{I}) & k = 3, & N = 1, & l = 2, & M \geqslant 2,\\
		\mathrm{II}) & k = 3, & N \geqslant 2, & l = 2, & M = 2, \\
		\mathrm{III}) & \text{any } k, & \text{any } N, & l = 2, & M = 1, \\
		\mathrm{I}') & k = 2, & N \geqslant 2, & l = 3, & M = 1,\\
		\mathrm{II}') & k = 2, & N = 2, & l = 3, & M \geqslant 2,\\
		\mathrm{III}') & k = 2, & N = 1, & \text{any } l, & \text{any } M, \\
	\end{array}
\end{equation}
where $N = \min(\underline n)$ and $M = \min(\underline m)$. We will refer to them as Kobayashi's conditions \eqref{eqn:kobcon}.

\begin{Remark}
	The subcase of $\mathrm{III}')$ with $P$ minimal parabolic, has already been studied extensively by P. Magyar, T. Hashimoto and M. Colarusso-S. Evens (see \cite{Magyar}, \cite{Hashimoto}, \cite{ColarussoEvens}).
\end{Remark}

\subsection{A Bruhat-type embedding}

\begin{Definition}{\label{def Q_H}}
	If $H$ is a subgroup of $G$, we let $\mathcal Q_H$ be the set of all (possibly non-standard) maximal parabolic subgroups of $G$ containing $H$, namely
	\begin{equation*}
		\mathcal Q_H = \{M_\tau P_{(p,q)}M_\tau^{-1} \mid p+q = n, \tau \in S_n, P_{(p,q)}^\tau \subset H\},
	\end{equation*}
	where we have used the notations $M_\tau$ for the matrix corresponding to a permutation $\tau \in \mathfrak S_n$.
\end{Definition}

Our first result classifies Kobayashi's conditions for which $B'$-orbits on $G/P_{\underline m}$ can be computed using rank functions.

% \begin{theoremx}\label{Theorem A}
% 	Assume that the pair $(G',P)$ satisfies one of Kobayashi's conditions \eqref{eqn:kobcon} {\color{red}A FAIRE}

% 	and let $\mathcal Q_{B'}$ and $\mathcal H$ be either $\{P\}$ or $\mathcal Q_P$.
	
% 	\begin{enumerate}[label = $\roman*)$]
% 		\item The map
% 		$$\Psi : B'\backslash G/P \to \displaystyle \prod_{L \in \mathcal Q_{B'}, H \in \mathcal Q_{P}} L\backslash G/H$$
% 		is injective.
		
% 		\item The image of $\Psi$ is the set of all $$(X_{L,H})_{(L,H) \in \mathcal Q_{B'} \times \mathcal Q_{P}} \in \displaystyle \prod_{L \in \mathcal Q_{B'}, H \in \mathcal Q_P} L\backslash G/H $$
% 		such that
% 		$$\bigcap_{L \in \mathcal Q_{B'}, H \in \mathcal Q_{P}} X_{L,H} \neq \emptyset$$
% 		and $\Psi^{-1}$ is given by
% 		$$(X_{L,H})_{(L,H) \in \mathcal \mathcal Q_{B'} \times \mathcal Q_{P}} \mapsto \bigcap_{L \in \mathcal Q_{B'}, H \in \mathcal Q_{P}} X_{L,H}.$$
% 	\end{enumerate}
% \end{theoremx}

\begin{theoremx}\label{Theorem A}
	Assume that $B'\backslash G/P_{\underline m}$ is finite and consider the map
	\begin{equation*}
		\Psi : \begin{array}{ccc}
		B' \backslash G/P_{\underline m} & \to & \displaystyle \prod_{Q_1 \in \mathcal Q_{B'}, Q_2 \in \mathcal Q_{P_{\underline m}}} Q_1 \backslash G/Q_2 \\
		B'gP_{\underline m} & \mapsto & (Q_1gQ_2)_{\substack{Q_1 \in \mathcal Q_{B'} \\ Q_2 \in \mathcal Q_{P_{\underline m}}}}
	\end{array}.
	\end{equation*}
	\begin{enumerate}[label = $\roman*)$]
		\item If Kobayashi's condition $\mathrm{I}')$ holds with $m_2 \neq 1$, $n \geqslant 5$ and either $m_1 = 1$ or $m_3 = 1$, or if Kobayashi's condition $\mathrm{II}')$ holds, then the map $\Psi$ is not injective.
		\item In all remaining cases, $\Psi$ is injective.
		\item When $\Psi$ is injective, its image is the set of all
		\begin{equation*}
			(X_{L,H})_{\substack{L \in \mathcal Q_{B'} \\ H \in \mathcal Q_{P_{\underline m}}}} \in \displaystyle \prod_{\substack{L \in \mathcal Q_{B'} \\ H \in \mathcal Q_{P_{\underline m}}}} L\backslash G/H
		\end{equation*}
		such that
		\begin{equation*}
			\bigcap_{\substack{L \in \mathcal Q_{B'} \\ H \in \mathcal Q_{P_{\underline m}}}} X_{L,H} \neq \emptyset
		\end{equation*}
		and $\Psi^{-1}$ is given by
		\begin{equation*}
			(X_{L,H})_{\substack{L \in \mathcal Q_{B'} \\ H \in \mathcal Q_{P_{\underline m}}}} \mapsto \bigcap_{\substack{L \in \mathcal Q_{B'} \\ H \in \mathcal Q_{P_{\underline m}}}} X_{L,H}.
		\end{equation*}
	\end{enumerate}
\end{theoremx}
\begin{proof} $ $
	\begin{enumerate}[label = $\roman*)$]
		\item See \Cref{cas 1'} and \Cref{cas 2'}.

		\item See \Cref{cas 0}, \Cref{cas 3'}, \Cref{cas 1}, \Cref{cas 2} and \Cref{cas 3}.

		\item First note that according to $i)$, if $g,x \in G$, one has
		\begin{align*}
			x \in B'gP_{\underline m} &\iff B'xP_{\underline m} = B'gP_{\underline m}\\
			&\iff \Psi(B'xP_{\underline m}) = \Psi(B'gP_{\underline m}) \\
			&\iff \forall L \in \mathcal Q_{B'}, \forall H \in \mathcal Q_{P_{\underline m}}, \hspace{3mm} LxH = LgH \\
			&\iff \forall L \in \mathcal Q_{B'}, \forall H \in \mathcal Q_{P_{\underline m}}, \hspace{3mm} x \in LgH
		\end{align*}
		so that
		\begin{equation*}
			B'gP_{\underline m} = \bigcap_{\substack{L \in \mathcal Q_{B'} \\ H \in \mathcal Q_{P_{\underline m}}}}LgH.
		\end{equation*}
		Thus, it sufficies to prove the assertion about the image of $\Psi$.
		
		Let $S$ be the set of all
		\begin{equation*}
			(X_{L,H})_{\substack{L \in \mathcal Q_{B'} \\ H \in \mathcal Q_{P_{\underline m}}}} \in \displaystyle \prod_{L \in \mathcal Q_{B'}, H \in \mathcal Q_P} L\backslash G/H 
		\end{equation*}
		such that
		\begin{equation*}
			\bigcap_{\substack{L \in \mathcal Q_{B'} \\ H \in \mathcal Q_{P_{\underline m}}}} X_{L,H} \neq \emptyset.
		\end{equation*}
		If $B'gP_{\underline m} \in B'\backslash G/P_{\underline m}$ then $g \in LgH$ for all $L \in \mathcal Q_{B'}$, $H \in \mathcal Q_{P_{\underline m}}$ so that $\Psi(B'gP_{\underline m}) \in S$.
		Thus, we have proven that $\Psi(B'\backslash G/P_{\underline m}) \subset S$.

		Let $(X_{L,H})_{L,H} \in S$. If $g \in \cap_{L,H} X_{L,H}$, one has
		\begin{equation*}
			\forall L \in \mathcal Q_{B'}, \forall H \in \mathcal Q_{P_{\underline m}}, \hspace{3mm} X_{L,H} = LgH.
		\end{equation*}
		Let $(f_{L,H}: G/H \to \N_{\geqslant 0})_{L \in \mathcal Q_{B'},H \in \mathcal Q_{P_{\underline m}}}$ be the set of rank functions which are $B'$-invariant (see \Cref{subsection invariant}). From \Cref{lemme_rank}, we get
		\begin{equation*}
			\forall L \in \mathcal Q_{B'}, \forall H \in \mathcal Q_{P_{\underline m}}, \hspace{3mm} f_{L,H}^{-1}(\{f_{L,H}(gH)\}) = LgH.
		\end{equation*}
		Moreover, according to both $i)$ and \Cref{lemme_rank}, one has
		\begin{equation*}
			B'xP_{\underline m} = \bigcap_{\substack{L \in \mathcal Q_{B'} \\ H \in \mathcal Q_{P_{\underline m}}}} f_{L,H}^{-1}(\{f_{L,H}(gH)\}) = \bigcap_{\substack{L \in \mathcal Q_{B'} \\ H \in \mathcal Q_{P_{\underline m}}}} LgH.
		\end{equation*}
		In particular, $\bigcap_{L,H} LgH$ is a $B'$-orbit of $G/P_{\underline m}$ and it belongs to $B'\backslash G/P_{\underline m}$.
		Since 
		\begin{equation*}
			\Psi\Big(\bigcap_{\substack{L \in \mathcal Q_{B'} \\ H \in \mathcal Q_{P_{\underline m}}}} LgH\Big) = (X_{L,H})_{L \in \mathcal Q_{B'},H \in \mathcal Q_{P_{\underline m}}} \in \Psi(B'\backslash G/ P_{\underline m}),
		\end{equation*}
		we conclude that $S \subset \Psi(B'\backslash G/ P_{\underline m})$ so that $S = \Psi(B'\backslash G/ P_{\underline m})$.
	\end{enumerate}

\end{proof}

% \begin{Remark}
% 	Note that the codomain of $\Psi$ is finite as it is a finite product of finite sets, according to the Bruhat decomposition.
% 	Thus, if $(G',P)$ does not satisfy any condition \eqref{eqn:kobcon}, the map $\Phi$ cannot be injective as $B'\backslash G/P$ is infinite.
% \end{Remark}

\begin{Remark}
	When $k = 2$, $\mathcal Q_{B'}$ consists of the $2n-2$ maximal parabolic subgroups of the following form:
	
	
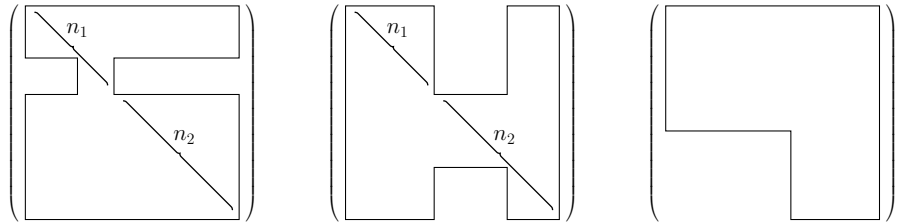
\vspace{1.5cm}
	
	\begin{center}
		\begin{tikzpicture}[transform canvas={scale=0.5},
			%Global config
			>=latex,
			line width=1pt,
			%Styles
			Brace/.style={
				decorate,
				decoration={
					brace,
					raise=-7pt
				}
			}
			]
			
			\matrix[% General option for all nodes
			matrix of nodes,
			text height=2.5ex,
			text depth=0.75ex,
			text width=3.25ex,
			align=center,
			left delimiter=(,
			right delimiter= ),
			column sep=5pt,
			row sep=5pt,
			%nodes={draw=black!10}, % Uncoment to see the square nodes.
			nodes in empty cells,
			] at (0,0) (M){ % Matrix contents  
				& & & & & \\
				& & & & & \\
				& & & & & \\
				& & & & & \\
				& & & & & \\
				& & & & & \\
			};
			% Drawing the sectors using matrix coordinate names.
			\draw[thick,draw] (M-1-1.north west)
			-- (M-1-6.north east)
			-- (M-2-6.east)
			-- (M-2-3.center)
			-- (M-3-3.center)
			-- (M-3-6.east)
			-- (M-6-6.south east)
			-- (M-6-1.south west)
			-- (M-3-1.west)
			-- (M-3-2.center)
			-- (M-2-2.center)
			-- (M-2-1.west)
			-- cycle;
			%			\draw[thick,draw] (M-1-1.north west)
			%			-- (M-1-3.north)
			%			-- (M-3-3.center)
			%			-- (M-3-2.center)
			%			-- (M-2-2.center)
			%			-- (M-2-1.west)
			%			-- cycle;
			%			\draw[thick,draw] (M-3-3.center)
			%			-- (M-3-6.east)
			%			-- (M-6-6.south east)
			%			-- cycle;
			%			\draw[thick,draw] (M-3-1.west)
			%			-- (M-3-3.center)
			%			-- (M-6-3.south)
			%			-- (M-6-1.south west)
			%			-- cycle;
			% Drawing the braces.   
			\draw[Brace] (M-1-1.north)
			-- (M-3-3.north)
			node[midway,above]{\LARGE $n_1$};
			
			\draw[Brace] (M-3-3.east)
			-- (M-6-6.east)
			node[midway,above]{\LARGE $n_2$};
		\end{tikzpicture} \hspace{4cm}
		\begin{tikzpicture}[transform canvas={scale=0.5},
			%Global config
			>=latex,
			line width=1pt,
			%Styles
			Brace/.style={
				decorate,
				decoration={
					brace,
					raise=-7pt
				}
			}
			]
			
			\matrix[% General option for all nodes
			matrix of nodes,
			text height=2.5ex,
			text depth=0.75ex,
			text width=3.25ex,
			align=center,
			left delimiter=(,
			right delimiter= ),
			column sep=5pt,
			row sep=5pt,
			%nodes={draw=black!10}, % Uncoment to see the square nodes.
			nodes in empty cells,
			] at (0,0) (M){ % Matrix contents  
				& & & & & \\
				& & & & & \\
				& & & & & \\
				& & & & & \\
				& & & & & \\
				& & & & & \\
			};
			% Drawing the sectors using matrix coordinate names.
			\draw[thick,draw] (M-1-1.north west)
			-- (M-1-3.north)
			-- (M-3-3.center)
			-- (M-3-5.center)
			-- (M-1-5.north)
			-- (M-1-6.north east)
			-- (M-6-6.south east)
			-- (M-6-5.south)
			-- (M-5-5.center)
			-- (M-5-3.center)
			-- (M-6-3.south)
			-- (M-6-1.south west)
			-- cycle;
			
			\draw[Brace] (M-1-1.north)
			-- (M-3-3.north)
			node[midway,above]{\LARGE $n_1$};
			
			\draw[Brace] (M-3-3.east)
			-- (M-6-6.east)
			node[midway,above]{\LARGE $n_2$};
		\end{tikzpicture} \hspace{4cm}
		\begin{tikzpicture}[transform canvas={scale=0.5},
			%Global config
			>=latex,
			line width=1pt,
			%Styles
			Brace/.style={
				decorate,
				decoration={
					brace,
					raise=-7pt
				}
			}
			]
			
			\matrix[% General option for all nodes
			matrix of nodes,
			text height=2.5ex,
			text depth=0.75ex,
			text width=3.25ex,
			align=center,
			left delimiter=(,
			right delimiter= ),
			column sep=5pt,
			row sep=5pt,
			%nodes={draw=black!10}, % Uncoment to see the square nodes.
			nodes in empty cells,
			] at (0,0) (M){ % Matrix contents  
				& & & & & \\
				& & & & & \\
				& & & & & \\
				& & & & & \\
				& & & & & \\
				& & & & & \\
			};
			% Drawing the sectors using matrix coordinate names.    
			\draw[thick,draw] (M-1-1.north west)
			-- (M-1-6.north east)
			-- (M-6-6.south east)
			-- (M-6-4.south)
			-- (M-4-4.center)
			-- (M-4-1.west)
			-- cycle;
			
%			\draw[Brace] (M-1-1.north)
%			-- (M-3-3.north)
%			node[midway,above]{$n_1$};
%			
%			\draw[Brace] (M-3-3.east)
%			-- (M-6-6.east)
%			node[midway,above]{$n_2$};
			\node at (3.5,0) {\LARGE .};
		\end{tikzpicture}
	\end{center}
	
	\vspace{1.5cm}
\end{Remark}

% \begin{Remark}
% 	When $\mathcal L = \{B'\}$ and $\mathcal H = \mathcal Q_{B'}$, \Cref{Theorem A} states that one may obtain the $B'$-orbits on $G/P$ by considering $B'$-orbits on simpler manifolds, namely Grassmannian manifolds.
% \end{Remark}
\begin{Remark}
	Let $\underline m = (m_1,\dots,m_l)$ be a partition of $n$, $B$ be the standard Borel subgroup of $GL(n,\R)$ and $B^*$ be the (non-standard) Borel subgroup of $GL(n,\R)$ defined as the isotropy subgroup of the full flag
	\begin{equation*}
		\Span(e_n) \subset \Span(e_n,e_1) \subset \cdots \subset \Span(e_n,e_1,\dots,e_k) \subset \cdots \subset \Span(e_n,e_1,\dots,e_{n-1}) = \R^n.
	\end{equation*}
	According to the Bruhat decomposition (see \Cref{Fact Bruhat}), the $B$-orbits on $GL(n,\R)/P_{\underline m}$ are determined by the functions $\rank_{J,s} : GL(n,\R)/P_{\underline m} \to \N_{\geqslant 0}$ for
	\begin{equation*}
		J \in \mathcal J \coloneq \Big\{\{n-i+1,n-i+2,\dots,n\} \mid 1\leqslant i \leqslant n\Big\}, \hspace{3mm} 1\leqslant s \leqslant l
	\end{equation*}
	% $$(\rank_{\{n-i+1,\dots,n\},j})_{1 \leqslant i,j \leqslant n}$$
	while the $B^*$-orbits on $G/B$ are determined by the functions $\rank_{J,s}$ for
	% $$(\rank_{\{n-i+1,\dots,n-1\},j}), \hspace{3mm} (i\rank_{\{1,\dots,n\},j}).$$
	\begin{equation*}
		J \in \mathcal J^* \coloneq \Big\{\{n-i+1,n-i+2,\dots,n-1\} \mid 1\leqslant i \leqslant n-1\Big\} \cup \Big\{\{1,\dots,n\}\Big\}, \hspace{3mm} 1\leqslant s \leqslant l.
	\end{equation*}
	Since $\rank_{J,s} : GL(n,\R)/P_{\underline m} \to \N_{\geqslant 0}$ is $B'$-invariant if and only if $J \in \mathcal J \cup \mathcal J^*$, we see that the map
	\begin{equation*}
		\begin{array}{ccc}
		B' \backslash GL(n,\R)/P_{\underline m} & \to & B\backslash GL(n,\R) /P_{\underline m} \times B^*\backslash GL(n,\R) /P_{\underline m} \\
		B'gB & \mapsto & (BgB,B^*gB)
	\end{array}
	\end{equation*}
	is injective. Using the same argument as in the proof of \Cref{Theorem A}, we find that the $B'$-orbits on $GL(n,\R)/P_{\underline m}$ are precisely the non-empty intersection of $B$-orbits with $B^*$-orbits.
	In particular, the \Cref{Theorem A} recovers results of \cite[Corollary 2.5]{ColarussoEvens} and provides their generalization for any flag manifold.
\end{Remark}

When $G' = GL(n-1,\R) \times GL(1,\R)$ and $P_{\underline m} = B$ the Borel subgroup of $GL(n,\R)$, it is known from \cite{ColarussoEvens} that $\Psi$ is in fact a poset isomorphism onto its image when its range is endowed with the direct product order.
In the general case, we know from the semi-lower continuity of the rank maps that $\Psi$ is an increasing map.

\begin{Conjecture*}
	Whenever $\Psi$ is injective, it induces a poset isomorphism onto its image.
\end{Conjecture*}

Our proof of \Cref{Theorem A} relies on the explicit computation of the $B'$-orbits on $G/P$. We will state this classification in the setting of case $\mathrm{III}')$ for $(n_1,n_2) = (n-1,1)$ and case $0)$ for general $\underline m$ even though we obtained analogous descriptions for all conditions \eqref{eqn:kobcon} for which $\Psi$ is injective (see \Cref{Proof}).

\subsection{Orbits description: multiplicity-free case}

Kobayashi's condition $\mathrm{III}')$ with $(n_1,n_2) = (n-1,1)$ is particularly interesting from a representation theoretic point of view as $(GL(n,\R),GL(n-1,\R))$ is a multiplicity-free pair (see \cite{zhu}).

\begin{theoremx}\label{Theorem B}
	Let $P_{\underline m}$ be any parabolic subgroup of $GL(n,\R)$ and $B'$ be the standard Borel subgroup of $GL(n-1,\R) \times GL(1,\R)$. We denote by $(e_1,\dots,e_{n-1})$ the canonical basis of $\R^{n-1}$. Then, any $B'$-orbit of $\mathcal D_{\underline{m}} \simeq G/P_{\underline m}$ admits a representative given by
	\begin{equation*}
		A = \begin{bmatrix}
		A_1 & \cdots & A_{l}
	\end{bmatrix}
	\end{equation*}
	where each block $A_j$ is of the form
	\begin{equation}\label{eqn:normalform1}
		\begin{pmatrix}
			e_{i_{j,1}} & \dots & e_{i_{j,m_j}} \\
			0 & \dots & 0
		\end{pmatrix}
	\end{equation}
	except exactly one of them which is of the form
	\begin{equation}\label{eqn:normalform2}
		\begin{pmatrix}
			\displaystyle \sum_{s=1}^k e_{i_s} & e_{i_{j,2}} & \dots & e_{i_{j,m_j}} \\
			1 & 0 & \cdots & 0
		\end{pmatrix}
	\end{equation}
	for some $k \in \{0,\dots,l-1\}$. Moreover,
	\begin{itemize}[leftmargin=*, label = --]
		\item all basis vectors $e_{i_{j,j'}}$ involved are distinct,
		\item if the block $A_{j_0}$ is of the form \eqref{eqn:normalform2}, then there exists an increasing sequence $j_0 < j_1 < \cdots < j_k$ such that $e_{i_s}$ appears in the block $A_{j_s}$ for all $s \in \{1,\dots,k\}$.
	\end{itemize}
	This normal form is unique up to permutations of the $e_{j,j'}$ vectors in each block $A_j$.
\end{theoremx}
%	
%	
%	where either :
%	\begin{itemize}[leftmargin=*]
%		\item each block $A_j$ is of the form
%		\begin{equation}\label{eqn:normalform1}
%			\begin{pmatrix}
%				e_{i_{j,1}} & \dots & e_{i_{j,m_j}} \\
%				0 & \dots & 0
%			\end{pmatrix}\tag{$*$}
%		\end{equation}
%		and all $(i_{j,j'})_{j,j'}$ are distinct,
%		
%		\item there exists a unique $t \in \{1,\dots,k-1\}$ such that $A_t$ is of the form
%		$$\begin{pmatrix}
%				0 & e_{i_{t,2}} & \dots & e_{i_{t,m_t}} \\
%				1 & 0 & \dots & 0
%			\end{pmatrix}$$
%		and all $A_j$ ($j \neq t$) are of the form \eqref{eqn:normalform1} and all $(i_{p,q})_{p,q}$ are distinct,
%		
%		\item there exists a unique $t \in \{1,\dots,k-1\}$ such that $A_t$ is of the form
%		\begin{equation}
%			\begin{pmatrix}
%				e_{i_{0}} & e_{i_{t,2}} & \dots & e_{i_{t,m_t}} \\
%				1 & 0 & \dots & 0
%			\end{pmatrix} \tag{$**$}
%		\end{equation}
%		and all $A_j$ ($j \neq t$) are of the form \eqref{eqn:normalform1} and all $(i_{j,j'})_{j,j'}$ are distinct and different from $i_0$ if $j \leqslant t$,
%		
%		\item there exists a unique $t \in \{1,\dots,k-1\}$ such that $A_t$ is of the form
%		\begin{equation}\label{eqn:normalform3}
%			\begin{pmatrix}
%				e_{i_{0}}+e_{i_{1}} & e_{i_{t,2}} & \dots & e_{i_{t,m_t}} \\
%				1 & 0 & \dots & 0
%			\end{pmatrix} \tag{$***$}
%		\end{equation}
%		with $i_1 < i_0$ and all $A_j$ ($j \neq t$) are of the form \eqref{eqn:normalform1}, all $(i_{j,j'})_{j,j'}$ are distinct, there exists $(t',j')$ with $t' > t$ such that $i_{t',j'} = i_0$ and $i_{j,j'} \neq i_1$ if $j \leqslant t'$.
%	\end{itemize}

\begin{corollaryx}{\label{orbites fermées multiplicity free}}
	In the setting of \Cref{Theorem B}, the closed $B'$-orbits coincide with the $B'$-fixed points. More precisely, they are given by the normal forms satisfying 
	\begin{enumerate}[leftmargin=*, label = $(\roman*)$]
		\item $k = 0$,
		\item $i_{j_1,j_1'} < i_{j_2,j_2'}$ whenever $j_1 < j_2$.
	\end{enumerate}
\end{corollaryx}
\begin{proof}
	First note that the normal forms satisfying both conditions $(i)$ and $(ii)$ do correspond to $B'$-fixed points, hence to closed orbits.
	
	On the other hand, if some normal form has $k \geqslant 1$, we may choose elements in the same orbit for which the block of type \eqref{eqn:normalform2} is replaced by
	\begin{equation*}
		\begin{pmatrix}
			\displaystyle \sum_{s=1}^k \frac{1}{N}e_{i_s} & e_{i_{j,2}} & \dots & e_{i_{j,m_j}} \\
			1 & 0 & \cdots & 0
		\end{pmatrix}
	\end{equation*}
	and then let $N \to +\infty$ to see that the orbit is not closed. In particular, the normal form of a closed orbit must satisfy condition $(i)$.
	
	Now assume that some normal form satisfies condition $(i)$ but not condition $(ii)$. Then, if $j_1,j_1',j_2,j_2'$ are indices satisfying
	\begin{equation*}
		i_{j_1,j_1'} > i_{j_2,j_2'}, \hspace{1cm} j_1 < j_2,
	\end{equation*}
	we may replace the columns
	\begin{equation*}
		\begin{pmatrix}
		e_{i_{j_1,j_1'}} & e_{i_{j_2,j_2'}} \\
		0 & 0
	\end{pmatrix}
	\end{equation*}
	by
	\begin{equation*}
		\begin{pmatrix}
		\frac{1}{N}e_{i_{j_1,j_1'}} - e_{i_{j_2,j_2'}} & e_{i_{j_1,j_1'}} \\
		0 & 0
	\end{pmatrix}
	\end{equation*}
	using both the left $B'$-action and the right $P_{\underline m}$-action, then let $N \to +\infty$ to find again that the orbit is not closed. Thus, any closed orbit must satisfy both conditions $(i)$ and $(ii)$.
\end{proof}

\begin{Remark}
	The closed orbits described in \Cref{orbites fermées multiplicity free} are exactly the double cosets of the matrices
	\begin{equation*}
		\begin{pmatrix}
		I_k & 0 & 0 \\
		0 & 0 & I_{n-k-1} \\
		0 & 1 & 0
	\end{pmatrix} \in GL(n,\R).
	\end{equation*}
	These $B'$-fixed points have been used in \cite{DitlevsenLabriet} to study differential symmetry breaking operators for the pair $(GL(n,\R),GL(n-1,\R))$.
\end{Remark}

\begin{corollaryx}
	In the setting of \Cref{Theorem B}, one has
	\begin{equation}\label{Corollary D}
		\left|B'\backslash G/P_{\underline m}\right| = \frac{1}{n}\binom{n}{m_1,\dots,m_l}\sum_{k=1}^{l}\frac{1}{(k-1)!}\sigma_k(m_1,\dots,m_l)
	\end{equation}
	where $\sigma_k$ is the $k$-th symmetric function in $l$-variables and 
	\begin{equation*}
		\binom{n}{m_1,\dots,m_l} = \frac{n!}{m_1!\cdots m_l!}
	\end{equation*}
	is a multinomial coefficient.
\end{corollaryx}
\begin{proof}
	Let us consider all the normal forms for which $k = 0$. These are in bijection with the quotient of the symmetric group $S_n$ by its subgroup $S_{m_1}\times \cdots \times S_{m_l}$ which is of cardinality equal to
	\begin{equation*}
		\binom{n}{m_1,\dots,m_l} = \frac{1}{n}\binom{n}{m_1,\dots,m_l}\sigma_1(m_1,\dots,m_l).
	\end{equation*}
	We now consider all the normal forms for a specific value of $k \geqslant 1$. Such a normal form is determined by the set $\{i_1,\dots,i_{k}\}$, the sequence $j_0 < \cdots < j_k$ and the sets $\{e_{i_{j,1}},\dots,e_{i_{j,m_j}}\}$ (or $\{e_{i_{j,2}},\dots,e_{i_{j,m_j}}\}$ if $j \in \{j_0,\dots,j_k\}$). The number of such normal forms equals
	\begin{equation*}
		\sum_{1 \leqslant j_0 < \cdots < j_k \leqslant l} \binom{n-1}{k}\frac{(n-1-k)!}{m_1!\cdots m_{l-1}!}m_{j_0}\cdots m_{j_k} \\ 
		= \frac{1}{n} \binom{n}{m_1,\dots,m_l} \frac{1}{k!}\sigma_{k+1}(m_1,\dots,m_l).
	\end{equation*}
	Thus, the formula (\ref{Corollary D}) follows.
\end{proof}

% \begin{Theorem}
% 	The map $\Psi$ induces a poset isomorphism onto its image {\color{red} A FAIRE}.
% \end{Theorem}

\begin{Example}
	The Hasse diagram for the pair $(G,G') = (GL(3,\R),GL(2,\R) \times GL(1,\R))$ and $P$ being minimal parabolic is given in \Cref{GL3-GL2}. Integers on the left indicate orbit dimensions.
	
	\begin{center}
		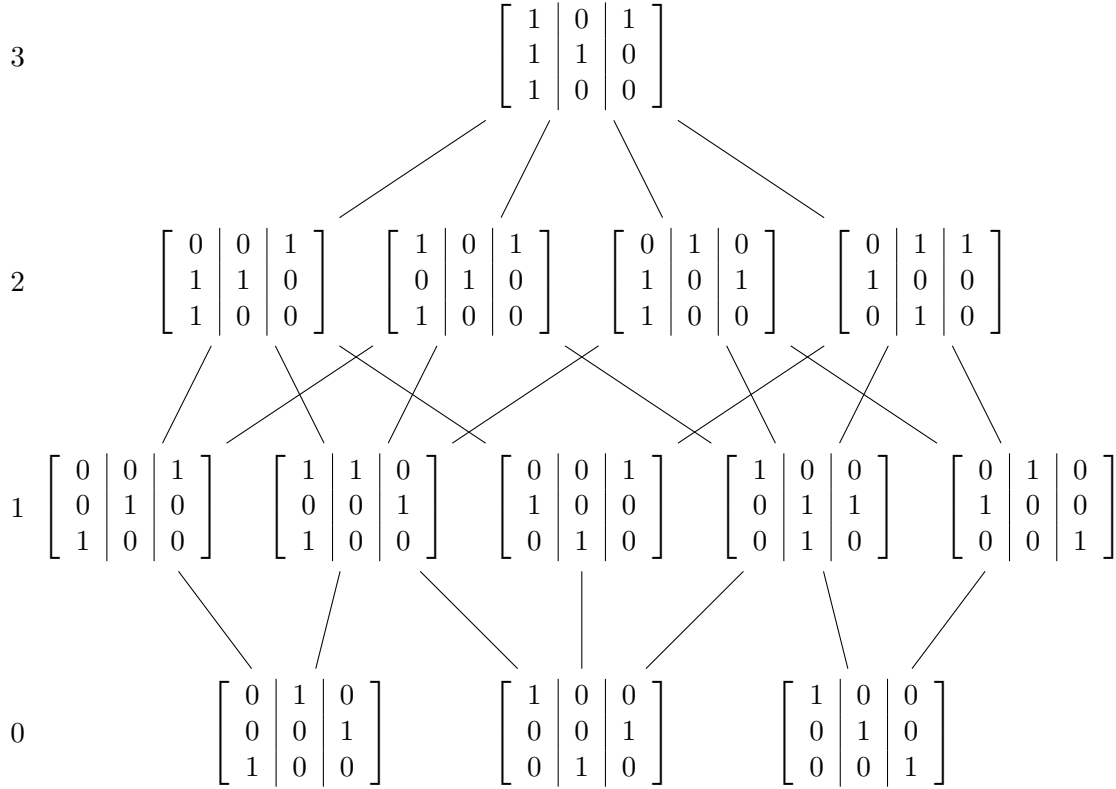
\begin{figure}[H]
			\centering
			\begin{tikzpicture}[scale=0.3]
			\node (dim0) at (-25,0) {$0$};
			\node (dim1) at (-25,10) {$1$};
			\node (dim2) at (-25,20) {$2$};
			\node (dim3) at (-25,30) {$3$};
				\node (X1) at (-12.5,0) {$\left[\begin{tabular}{c|c|c}
					0 & 1 & 0 \\
					0 & 0 & 1\\
					1 & 0 & 0
				\end{tabular}\right]$};
				\node (X2) at (0,0) {$\left[\begin{tabular}{c|c|c}
					1 & 0 & 0\\
					0 & 0 & 1\\
					0 & 1 & 0
				\end{tabular}\right]$};
				\node (X3) at (12.5,0) {$\left[\begin{tabular}{c|c|c}
					1 & 0 & 0\\
					0 & 1 & 0\\
					0 & 0 & 1
				\end{tabular}\right]$};
				\node (Y1) at (-20,10) {$\left[\begin{tabular}{c|c|c}
					0 & 0 & 1\\
					0 & 1 & 0\\
					1 & 0 & 0
				\end{tabular}\right]$};
				\node (Y2) at (-10,10) {$\left[\begin{tabular}{c|c|c}
					1 & 1 & 0\\
					0 & 0 & 1\\
					1 & 0 & 0
				\end{tabular}\right]$};
				\node (Y3) at (-0,10) {$\left[\begin{tabular}{c|c|c}
					0 & 0 & 1\\
					1 & 0 & 0\\
					0 & 1 & 0
				\end{tabular}\right]$};
				\node (Y4) at (10,10) {$\left[\begin{tabular}{c|c|c}
					1 & 0 & 0\\
					0 & 1 & 1\\
					0 & 1 & 0
				\end{tabular}\right]$};
				\node (Y5) at (20,10) {$\left[\begin{tabular}{c|c|c}
					0 & 1 & 0\\
					1 & 0 & 0\\
					0 & 0 & 1
				\end{tabular}\right]$};
				\node (Z1) at (-15,20) {$\left[\begin{tabular}{c|c|c}
					0 & 0 & 1\\
					1 & 1 & 0\\
					1 & 0 & 0
				\end{tabular}\right]$};
				\node (Z2) at (-5,20) {$\left[\begin{tabular}{c|c|c}
					1 & 0 & 1\\
					0 & 1 & 0\\
					1 & 0 & 0
				\end{tabular}\right]$};
				\node (Z3) at (5,20) {$\left[\begin{tabular}{c|c|c}
					0 & 1 & 0\\
					1 & 0 & 1\\
					1 & 0 & 0
				\end{tabular}\right]$};
				\node (Z4) at (15,20) {$\left[\begin{tabular}{c|c|c}
					0 & 1 & 1\\
					1 & 0 & 0\\
					0 & 1 & 0
				\end{tabular}\right]$};
				\node (T) at (0,30) {$\left[\begin{tabular}{c|c|c}
					1 & 0 & 1\\
					1 & 1 & 0\\
					1 & 0 & 0
				\end{tabular}\right]$};
				
				\draw[black] (X1) -- (Y1);
				\draw[black] (X1) -- (Y2);
				\draw[black] (X2) -- (Y2);
				\draw[black] (X2) -- (Y3);
				\draw[black] (X2) -- (Y4);
				\draw[black] (X3) -- (Y4);
				\draw[black] (X3) -- (Y5);
				\draw[black] (Y1) -- (Z1);
				\draw[black] (Y1) -- (Z2);
				\draw[black] (Y2) -- (Z1);
				\draw[black] (Y2) -- (Z2);
				\draw[black] (Y2) -- (Z3);
				\draw[black] (Y3) -- (Z1);
				\draw[black] (Y3) -- (Z4);
				\draw[black] (Y4) -- (Z2);
				\draw[black] (Y4) -- (Z3);
				\draw[black] (Y4) -- (Z4);
				\draw[black] (Y5) -- (Z3);
				\draw[black] (Y5) -- (Z4);
				\draw[black] (Z1) -- (T);
				\draw[black] (Z2) -- (T);
				\draw[black] (Z3) -- (T);
				\draw[black] (Z4) -- (T);
			\end{tikzpicture}
			\caption{Hasse diagram for $(G,G') = (GL(3,\R),GL(2,\R) \times GL(1,\R))$ and $P$ minimal parabolic.}
			\label{GL3-GL2}
		\end{figure}
	\end{center}
\end{Example}

\subsection{Orbits description: grassmannian case}

% Whenever both partitions $\underline n$ and $\underline m$ have length $2$, both flag manifolds $G/P_{\underline n}$ and $G/P_{\underline m}$ are Riemannian symmetric. As in T. Kobayashi's classification of visible actions (\cite{ActionVisible}), we will use this symmetric case to compute the orbits in the six other cases (see \Cref{Proof}).

In \cite{ActionVisible}, T. Kobayashi uses the special case where both partitions $\underline n$ and $\underline m$ have length $2$ as a starting point to treat the remaining cases.
We will use the same strategy to compute the orbits in the six other cases (see \Cref{Proof}).

\setcounter{thmxx}{1}
\begin{thmxx}\label{Theorem B'}
	Let $P_{\underline m}$ be the maximal parabolic subgroup of $GL(n,\R)$ associated to a composition $\underline m = (m_1,m_2)$ of $n$ and let $B'$ be the standard Borel subgroup of $GL(n_1,\R) \times GL(n_2,\R)$. We denote by $(e_1,\dots,e_{n_1})$ the canonical basis of $\R^{n_1}$ and $(f_1,\dots,f_{n_2})$ the canonical basis of $\R^{n_2}$. Then, any $B'$-orbit of $\mathcal D_{\underline{m}} \simeq M_{n,m_1}(\R)_{\mathrm{reg}}/P_{(m_1)}$ admits a representative given by
	\begin{equation*}
		\begin{bmatrix}
		0 & \cdots & 0  & e_{j_{r+1}} & \cdots & e_{j_{s}} & e_{j_{s+1}} & \cdots & e_{j_{m_1}} \\
		f_{i_1} & \cdots & f_{i_{r}}  & f_{i_{r+1}} & \cdots & f_{i_{s}} & 0 & \cdots & 0
	\end{bmatrix}.
	\end{equation*}
	Such a representative is unique if we assume that
	\begin{itemize}[leftmargin=*, label = --]
		\item $i_1 < \cdots < i_{s}$,
		\item all indices $j_t$ are different, $j_{s+1} < \cdots < j_{m_1}$.
	\end{itemize}
\end{thmxx}

\setcounter{corollaryxx}{2}
\begin{corollaryxx}
	In the setting of \Cref{Theorem B'}, the closed $B'$-orbits coincide with the $B'$-fixed points. More explicitly, they are given by the normal forms
	\begin{equation*}
		\begin{bmatrix}
		0 & \cdots & 0 & e_{1} & \cdots & e_{m_1-r} \\
		f_1 & \cdots & f_r & 0 & \cdots & 0
	\end{bmatrix}.
	\end{equation*}
	for $r \in \{\max(0,m_1-n_1),\dots,\min(m_1,n_1)\}$.
\end{corollaryxx}
\begin{proof}
	First observe that any flag of the form
	\begin{equation*}
		\begin{bmatrix}
		0 & \cdots & 0 & e_{1} & \cdots & e_{m_1-r} \\
		f_1 & \cdots & f_r & 0 & \cdots & 0
	\end{bmatrix}
	\end{equation*}
	is fixed under the $P'$-action. In particular, its orbit is a singleton, hence closed. \\
	Conversely, let $\mathcal O$ be a closed $P'$-orbit having the following normal form
	\begin{equation*}
		D = \begin{bmatrix}
		0 & \cdots & 0  & e_{j_{r+1}} & \cdots & e_{j_{s}} & e_{j_{s+1}} & \cdots & e_{j_{m_1}} \\
		f_{i_1} & \cdots & f_{i_{r}}  & f_{i_{r+1}} & \cdots & f_{i_{s}} & 0 & \cdots & 0
	\end{bmatrix}.
	\end{equation*}
	If $s \neq r$, one has	
	\begin{multline*}
		\mathcal O \ni \begin{bmatrix}
			0 & \cdots & 0  & \frac{1}{N} e_{j_{r+1}} & e_{j_{r+2}} & \cdots & e_{j_{s}} & e_{j_{s+1}} & \cdots & e_{j_{m_1}} \\
			f_{i_1} & \cdots & f_{i_{r}}  & f_{i_{r+1}} & f_{i_{r+2}} & \cdots & f_{i_{s}} & 0 & \cdots & 0
		\end{bmatrix} \\
		\xrightarrow[N \to +\infty]{}\begin{bmatrix}
			0 & \cdots & 0  & 0 & e_{j_{r+2}} & \cdots & e_{j_{s}} & e_{j_{s+1}} & \cdots & e_{j_{m_1}} \\
			f_{i_1} & \cdots & f_{i_{r}}  & f_{i_{r+1}} & f_{i_{r+2}} & \cdots & f_{i_{s}} & 0 & \cdots & 0
		\end{bmatrix}.
	\end{multline*}
	Since this is a normal form different from the one defining $\mathcal O$, this contradict $\mathcal O$ being closed. We conclude that we must have $s = r$ so that
	\begin{equation*}
		D = \begin{bmatrix}
		0 & \cdots & 0  & e_{j_{r+1}} & \cdots & e_{j_{m_1}} \\
		f_{i_1} & \cdots & f_{i_{r}}  & 0 & \cdots & 0
	\end{bmatrix}.
	\end{equation*}
	Since $i_1 < \cdots < i_r$, one has $i_t > t$ for all $t$ so that
	\begin{align*}
		D &\sim \begin{bmatrix}
			0 & \cdots & 0  & e_{j_{r+1}} & \cdots & e_{j_{m_1}} \\
			f_1 + \frac{1}{N}f_{i_1} & \cdots & f_r + \frac{1}{N}f_{i_{r}}  & 0 & \cdots & 0
		\end{bmatrix} \\
		& \xrightarrow[N \to +\infty]{} \begin{bmatrix}
			0 & \cdots & 0  & e_{j_{r+1}} & \cdots & e_{j_{m_1}} \\
			f_1 & \cdots & f_r & 0 & \cdots & 0
		\end{bmatrix}.
	\end{align*}
	Similarly, $\mathcal O$ being closed, unicity of normal forms implies that $(i_1,\dots,i_r) = (1,\dots,r)$. One may show in the same way that $(j_{r+1},\dots,j_{m_1}) = (1,\dots,k-r)$.
\end{proof}

\begin{Example}
	The Hasse diagram for the pair $(G,G') = (GL(4,\R),GL(2,\R) \times GL(2,\R))$ and $(m_1,m_2)=(1,3)$ is given in \Cref{hasse_diagram2}. Integers on the left indicate orbit dimensions.
	
	\begin{figure}[H]
		\centering
		\begin{tikzpicture}[scale=0.3]
			\node (dim0) at (-25,0) {$0$};
			\node (dim1) at (-25,10) {$1$};
			\node (dim2) at (-25,20) {$2$};
			\node (dim3) at (-25,30) {$3$};
			\node (E1) at (-7.5,0) {$\left[\begin{tabular}{c}
					0\\
					0\\
					1\\
					0
				\end{tabular}\right]$};
			\node (E2) at (-15,10) {$\left[\begin{tabular}{c}
					0\\
					0\\
					0\\
					1
				\end{tabular}\right]$};
			\node (F1) at (7.5,0) {$\left[\begin{tabular}{c}
					1\\
					0\\
					0\\
					0
				\end{tabular}\right]$};
			\node (F2) at (15,10) {$\left[\begin{tabular}{c}
					0\\
					1\\
					0\\
					0
				\end{tabular}\right]$};
			\node (EF11) at (0,10) {$\left[\begin{tabular}{c}
					1\\
					0\\
					1\\
					0
				\end{tabular}\right]$};
			\node (EF12) at (7.5,20) {$\left[\begin{tabular}{c}
					0\\
					1\\
					1\\
					0
				\end{tabular}\right]$};
			\node (EF21) at (-7.5,20) {$\left[\begin{tabular}{c}
					1\\
					0\\
					0\\
					1
				\end{tabular}\right]$};
			\node (EF22) at (0,30) {$\left[\begin{tabular}{c}
					0\\
					1\\
					0\\
					1
				\end{tabular}\right]$};
			\draw[black] (F1) -- (F2);
			\draw[black] (F1) -- (EF11);
			\draw[black] (E1) -- (E2);
			\draw[black] (E1) -- (EF11);
			\draw[black] (F2) -- (EF12);
			\draw[black] (E2) -- (EF21);
			\draw[black] (EF11) -- (EF12);
			\draw[black] (EF11) -- (EF21);
			\draw[black] (EF12) -- (EF22);
			\draw[black] (EF21) -- (EF22);
		\end{tikzpicture}
		\caption{Hasse diagram for $(G,G') = (GL(4,\R),GL(2,\R) \times GL(2,\R))$ and $(m_1,m_2)=(1,3)$.}
		\label{hasse_diagram2}
	\end{figure}
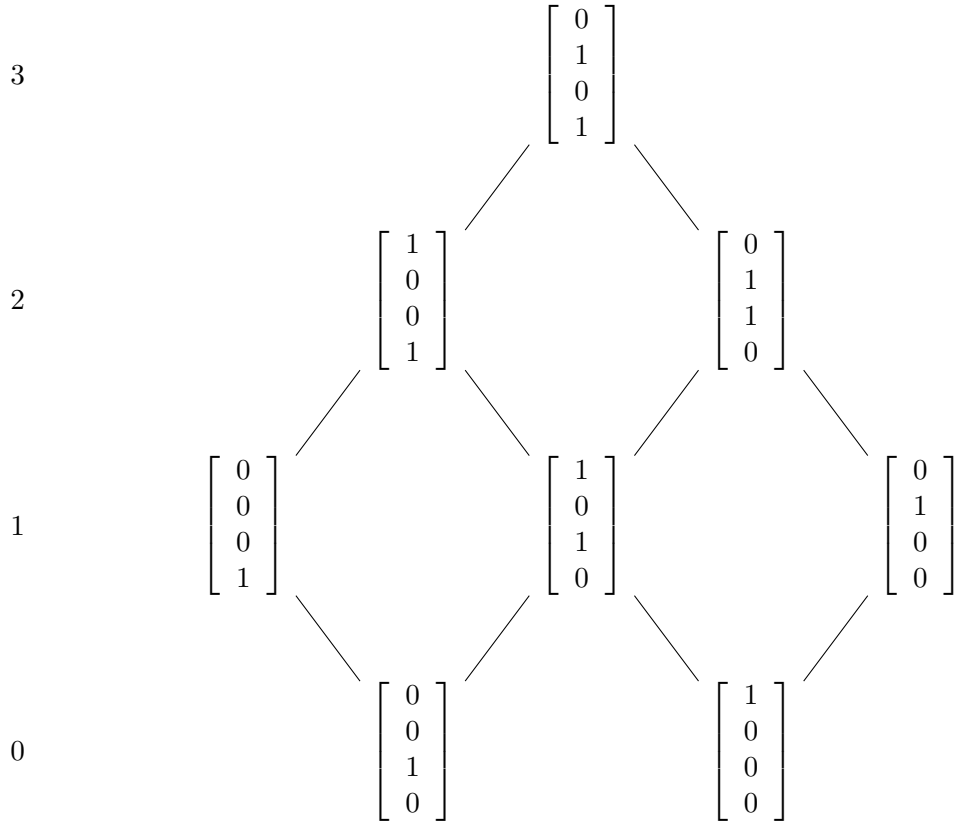
\end{Example}

\section{Proofs and remaining cases} \label{Proof}

In this section, we study the double coset space $B'\backslash G/P$ assuming Kobayashi's conditions \eqref{eqn:kobcon} one at a time in order to prove \cref{Theorem A}. The injectivity of the map $\Psi$ (see \Cref{Theorem A}) is proven in two steps:
\begin{itemize}[label = --]
	\item we first find a set of representatives for the $B'$-orbits on $G/P$,
	\item we then show that the $(B' \times P_{\underline m})$-invariant ranks introduced in \cref{subsection invariant} separate these representatives.
\end{itemize}
%\subsection[Computation of orbits]{Computation of $B'$-orbits}

\begin{Notation*}
	Let $\T(p,\R)$ denote the group of upper triangular matrices in $GL(p,\R)$ for $p \in \N_{>0}$.
	If two flags $D_1,D_2$ are in the same $B'$-orbit, we will write $D_1 \sim D_2$.
\end{Notation*}

\subsection[Case 0)]{Case $0)$} \label{cas 0}

Let $\underline m = (m_1,m_2), \underline n = (n_1,n_2)$ be compositions of $n$ and
\begin{equation*}
	B' = \left\{\begin{pmatrix}
	a & 0 \\
	0 & b
\end{pmatrix} \mid a \in \T(n_1,\R), \, b \in \T(n_2,\R)\right\} \subset GL(n,\R).
\end{equation*}
Following the $(M_{n,n-n_k}(\R)_{\mathrm{reg}}/P_{(n_1,\dots,n_{k-1})})$-realization of $\D_{\underline n}$ from \Cref{flag varieties}, will write any $D \in \D_{\underline m}$ as
\begin{equation*}
	D = \begin{bmatrix}
	U_{1} & \cdots & U_{m_1} \\
	V_{1} & \cdots & V_{m_1}
\end{bmatrix}
\end{equation*}
with $U_{1},\dots,U_{m_1} \in \R^{n_1}$ and $V_{1},\dots,V_{m_1} \in \R^{n_2}$.
Note that in this setting, $\rank_J$ is $B'$-invariant if and only if
\begin{multline*}
	J \in \mathcal J \coloneq \Big\{\{i,\dots,n_1\} \cup \{n_1+j,\dots,n\} \mid 1 \leqslant i \leqslant n_1, 1 \leqslant j \leqslant n_2\Big\} \\
	\cup \Big\{\{i,\dots,n_1\} \mid 1 \leqslant i \leqslant n_1\Big\} \cup \Big\{\{n_1+j,\dots,n\} \mid 1 \leqslant j \leqslant n_2\Big\}.
\end{multline*}
In particular, one has $\mathcal Q'_{B'} = \{P_J \mid J \in \mathcal J\}$ (see \Cref{def Q_H}).

The proof of \Cref{Theorem B'} relies on applying a suitable Gauss-Jordan elimination type process. The next lemma contains most of the procedure.

\begin{Lemma}{\label{lemme_fondamental}}
	Let $p,q \in \N_{>0}$ and $(e_1,\dots,e_q)$ be the canonical basis of $\R^q$. Consider the $T(p,\R) \times GL(q,\R)$ action on $M_{p,q}(\R)$ by
	\begin{equation*}
		(A,B)\cdot C = ACB^{-1}.
	\end{equation*}
	Then each element of $M_{p,q}(\R)$ belongs to a $(T(p,\R) \times GL(q,\R))$-orbit of the form
	\begin{equation*}
		(T(p,\R) \times GL(q,\R)) \cdot \begin{pmatrix}
		e_{i_1} & \cdots & e_{i_r} & 0 & \cdots & 0
	\end{pmatrix}
	\end{equation*}
	for some $r \in \N$ and $i_1 < \cdots < i_r$.
\end{Lemma}
\begin{proof}
	Let $A = (a_{ij}) \in M_{p,q}(\R)$. Since the $GL(q,\R)$-orbits on $M_{p,q}(\R)$ are determined by the linear span of columns of matrices, we may assume that $a_{ij} = 0$ if $j > \rank A \eqcolon r$.\\
	Next, we define $i_r$ as
	\begin{equation*}
		i_r = \max\{i \in \{1,\dots,p\} \mid \exists j \in \{1,\dots, q\}, \,  a_{ij} \neq 0\}.
	\end{equation*}
	Acting on the right by a permutation matrix if needed, we may assume that $a_{i_r,r} \neq 0$.
	Then, multiplying on the left by a suitable element of $T(p,\R)$ and on the right by a suitable element of $GL(q,\R)$, we find $B = (b_{ij})$ in the same orbit as $A$  which satisfies
	\begin{equation*}
		b_{ij} = \begin{cases}
		0 \text{ if }j > r \text{ or } \Big(j = r \text{ and } i \neq i_r \Big) \text{ or } \Big( i = i_r \text{ and } j \neq r \Big),\\
		1 \text{ if } (i,j) = (i_r,r).
	\end{cases}
	\end{equation*}
	We now proceed by finite induction on $r$ by defining 
	\begin{equation*}
		i_{r-1} = \max\{i \in \{1,\dots,i_{r-1}\} \mid \exists j \in \{1,\dots, r-1\}, \,  a_{ij} \neq 0\}
	\end{equation*}
	and applying the same process.
\end{proof}

\begin{Lemma}{\label{lemme thm B}}
	Let
	\begin{equation*}
		D = \begin{bmatrix}
		U_{1} & \cdots & U_{m_1} \\
		V_{1} & \cdots & V_{m_1}
	\end{bmatrix} \in \D_{\underline m}
	\end{equation*}
	with $U_{1},\dots,U_{m_1} \in \R^{n_1}$ and $V_{1},\dots,V_{m_1} \in \R^{n_2}$.
	There exist integers $r,s \geqslant 0$ and indices $i_1,\dots,i_s$, $j_{s+1},\dots,j_{m_1}$ satisfying
	\begin{equation*}
		D \sim \begin{bmatrix}
		0 & \cdots & 0 & U_{r+1} & \cdots & U_{s} & e_{j_{s+1}} & \cdots & e_{j_{m_1}} \\
		f_{i_1} & \cdots & f_{i_r} & f_{i_{r+1}} & \cdots & f_{i_s} & 0 & \cdots & 0
	\end{bmatrix}
	\end{equation*}
	with linearly independent $U_{r+1},\dots,U_s$.
\end{Lemma}
\begin{proof}
	First, we use the left $\T(n_2,\R)$-action and the right $GL(m_1,\R)$-action on the $V_1$ vectors to write, according to \Cref{lemme_fondamental}:
	\begin{equation*}
		D \sim \begin{bmatrix}
		U_{1} & \cdots & U_{s} & U_{s+1} & \cdots & U_{m_1} \\
		f_{i_1} & \cdots & f_{i_s} & 0 & \cdots & 0
	\end{bmatrix}
	\end{equation*}
	where $s = \rank(V_1,\dots,V_{m_1})$. Similarly, considering the left $\T(n_1,\R)$-action and the right $GL(m_1,\R)$-action, we obtain
	\begin{equation*}
		D \sim \begin{bmatrix}
		U_{1} & \cdots & U_{r} & e_{j_{r+1}} & \cdots & e_{j_{m_1}} \\
		f_{i_1} & \cdots & f_{i_r} & 0 & \cdots & 0
	\end{bmatrix}
	\end{equation*}
	because $D$ has dimension $m_1$. Using the right $GL(m_1,\R)$-action, we can also assume that $U_1,\dots,U_r$ all have zero $e_{j_{r+1}},\dots,e_{j_{m_1}}$ coordinates.\\
	Using Gaussian elimination, we may extract from $(U_1,\dots,U_s)$ a maximal linearly independent family $(U_{k_1},\dots,U_{k_p})$  satisfying the following property:
	\begin{equation*}
		\text{if } U_q = \sum_{j=1}^p x_{qj}U_{k_j} \text{ then } x_{qj} \neq 0 \implies k_{j} \leqslant i_q \leqslant i_j.
	\end{equation*}
%	{\color{red}(To obtain such a family, assume $i_1 > \cdots > i_s$ and apply the standard algorithm to extract a basis out of $(U_{1},\dots,U_r)$)} \\
	Then, for any $q \neq k_1,\dots,k_p$, we may replace the column
	\begin{equation*}
		\begin{pmatrix}
		U_q \\
		f_{i_q}
	\end{pmatrix} \hspace{3mm} \text{ by } \hspace{3mm} 
		\begin{pmatrix}
		0 \\
		f_{i_q} - \displaystyle\sum_{j=1}^p x_{qj}f_{i_{k_j}}
	\end{pmatrix}
	\end{equation*}
	using the right $GL(m_1,\R)$-action.
	Since $x_{qj} \neq 0$ implies $i_{k_j} > i_q$, we may then replace this column by $\,^t\begin{pmatrix}
		0 & f_{i_q}
	\end{pmatrix}$ using the left $B'$-action and the lemma follows.
\end{proof}

\begin{Lemma}{\label{lemme thm B bis}}
	For any $t \in \{0,\dots,s-r\}$, one has
	\begin{equation*}
		D \sim \begin{bmatrix}
		0 & \cdots & 0 & e_{j_{r+1}} & \cdots & e_{j_{r+t}} & U_{r+t+1} & \cdots & U_{s} & e_{j_{s+1}} & \cdots & e_{j_{m_1}} \\
		f_{i_1} & \cdots & f_{i_r} & f_{i_{r+1}} & \cdots & f_{i_{r+t}} & f_{i_{r+t+1}} & \cdots & f_{i_s} & 0 & \cdots & 0
	\end{bmatrix}.
	\end{equation*}
\end{Lemma}
\begin{proof}
	We proceed by finite induction on $t$. The initial case $t = 0$ has already been treated in \Cref{lemme thm B}.
	
	Assume that
	\begin{equation*}
		D \sim \begin{bmatrix}
		0 & \cdots & 0 & e_{j_{r+1}} & \cdots & e_{j_{r+t}} & U_{r+t+1} & \cdots & U_{s} & e_{j_{s+1}} & \cdots & e_{j_{m_1}} \\
		f_{i_1} & \cdots & f_{i_r} & f_{i_{r+1}} & \cdots & f_{i_{r+t}} & f_{i_{r+t+1}} & \cdots & f_{i_s} & 0 & \cdots & 0
	\end{bmatrix}
	\end{equation*}
	for some $t \in \{0,\dots,s-r-1\}$ and define
	\begin{equation*}
		m \coloneq \max \{k \mid 1 \leqslant k \leqslant n-1, \, \text{ some } U_j \text{ has a non-zero } e_k \text{ component}\},
	\end{equation*}
	\begin{equation*}
		j_0 \coloneq \min \{j \mid r+t+1 \leqslant j \leqslant s, \, U_{j} \text{ has a non-zero } e_m \text{ component}\}.
	\end{equation*}
	Then, one has
	\begin{equation*}
		U_{j_0} = \,^t \begin{pmatrix}
		\ast & \cdots & \ast & x & 0 & \cdots & 0
	\end{pmatrix}
	\end{equation*}
	with non-zero $m$-th entry $x$. Using the left $B'$-action, we may replace $U_{j_0}$ by $e_{m}$. If some $U_{p}$ has a non-zero $m$-th entry $y$, we may replace
	\begin{equation*}
		\begin{pmatrix}
		U_p \\
		f_{i_p}
	\end{pmatrix} \hspace{3mm} \text{ by } \hspace{3mm} 
		\begin{pmatrix}
		U_p-ye_m \\
		f_{i_p}-yf_{i_{j_0}}
	\end{pmatrix}
	\end{equation*}
	so that $U_p-ye_m$ has zero $m$-th entry. Since $j_0 \leqslant i_p$, we may use the left $B'$-action to replace this column by $\,^t \begin{pmatrix}
		U_p-ye_m & f_{i_p}
	\end{pmatrix}$.
	If we define $j_{r+t+1} \coloneq m$, $i_{r+t+1} \coloneq i_{j_0}$ and denote by $U_p$ the column vector $U_p-ye_m$, we obtain
	\begin{equation*}
		D \sim \begin{bmatrix}
		0 & \cdots & 0 & e_{j_{r+1}} & \cdots & e_{j_{r+t+1}} & U_{r+t+2} & \cdots & U_{s} & e_{j_{s+1}} & \cdots & e_{j_{m_1}} \\
		f_{i_1} & \cdots & f_{i_r} & f_{i_{r+1}} & \cdots & f_{i_{r+t+1}} & f_{i_{r+t+2}} & \cdots & f_{i_s} & 0 & \cdots & 0
	\end{bmatrix}
	\end{equation*}
	and the lemma is proved.
\end{proof}

\begin{Proposition}{\label{normal form exist}}
	For any flag $D \in \D_{\underline m}$, there exists $r,s \geqslant 0$ and pairwise distinct indices $i_1, \dots, i_{s}$, $j_{r+1}, \dots, j_{m_1}$ such that
	\begin{equation*}
		D \sim \begin{bmatrix}
		0 & \cdots & 0  & e_{j_{r+1}} & \cdots & e_{j_{s}} & e_{j_{s+1}} & \cdots & e_{j_{m_1}} \\
		f_{i_1} & \cdots & f_{i_{r}}  & f_{i_{r+1}} & \cdots & f_{i_s} & 0 & \cdots & 0
	\end{bmatrix}
	\end{equation*}
\end{Proposition}
\begin{proof}
	This follows from \Cref{lemme thm B bis} for $t = s-r$.
\end{proof}

\begin{Theorem}{\label{normal form unique}}
	The flags described in \Cref{normal form exist} have different invariant ranks. In particular, the normal form of \Cref{normal form exist} is unique and the map
	\begin{equation*}
		\Psi : B'\backslash G/P_{\underline m} \to \displaystyle \prod_{L \in \mathcal Q_{B'}, H \in \mathcal Q_{P_{\underline m}}} L\backslash G/H
	\end{equation*}
	is injective.
\end{Theorem}
% \begin{Proposition}
% 	The normal form described in \Cref{normal form exist} is unique up to permutation of columns.
% \end{Proposition}
\begin{proof}
	Let us consider a flag $D \in \D_{\underline m}$ described as
	\begin{equation*}
		D = \begin{bmatrix}
		0 & \cdots & 0  & e_{j_{r+1}} & \cdots & e_{j_{s}} & e_{j_{s+1}} & \cdots & e_{j_{m_1}} \\
		f_{i_1} & \cdots & f_{i_{r}}  & f_{i_{r+1}} & \cdots & f_{i_s} & 0 & \cdots & 0
	\end{bmatrix}.
	\end{equation*}
	First, note that 
	\begin{equation*}
		s = \rank_{\{n_1+1,\dots,n\}}D, \hspace{3mm} r = m_1-\rank_{\{1,\dots,n_1\}}D.
	\end{equation*}
	Next, the indices $i_1,\dots,i_s$ are obtained as the integers $p \in \{1,\dots,n_2\}$ satisfying
	\begin{equation*}
		\rank_{\{n_1+p,\dots,n\}}D - \rank_{\{n_1+p+1,\dots,n\}}D = 1
	\end{equation*}
	and the indices $j_{{r}+1},\dots,j_{m_1}$ are obtained as the integers $p \in \{1,\dots,n_1\}$ satisfying
	\begin{equation*}
		\rank_{\{p,\dots,n_1\}}D - \rank_{\{p+1,\dots,n_1\}}D = 1.
	\end{equation*}
	Finally, the column $\,^t\begin{pmatrix}
		e_j & f_i
	\end{pmatrix}$
	appears in the normal form if and only if
	\begin{equation*}
		\begin{cases}
		\rank_{\{j,\dots,n_1\}\cup \{n_1+i,\dots,n\}}D - \rank_{\{j+1,\dots,n_1\} \cup \{n_1+i,\dots,n\}}D = 0, \\
		\rank_{\{j,\dots,n_1\}\cup \{n_1+i+1,\dots,n\}}D - \rank_{\{j+1,\dots,n_1\} \cup \{n_1+i+1,\dots,n\}}D = 1.
	\end{cases}
	\end{equation*}
	The theorem follows from these conditions and the fact that these rank functions are $B'$-invariant.
\end{proof}

% \begin{Theorem}
% 	The map $$\begin{array}{ccc}
% 		B'\backslash G/P_{\underline m} & \to & \displaystyle \prod_{Q \in \mathcal Q'} Q\backslash G/P_{\underline m}
% 	\end{array}$$
% 	is injective and induces a poset isomorphism onto its image.
% \end{Theorem}
% \begin{proof}
% 	According to \Cref{lemme_rank}, it suffices to prove that 
% 	\begin{enumerate}[label = (\roman*)]
% 		\item any orbit is determined by its ranks,
% 		\item $X_1 \succ X_2 \iff \forall J \in \mathcal J, \, \rank_{J}X_1 \leqslant \rank_{J}X_2$.
% 	\end{enumerate}
% 	Property $(i)$ amounts to the proof of \Cref{normal form unique} and the direct implication of $(ii)$ is automatic since our rank maps are lower semi-continuous.
	
% 	\noindent Let $X_1,X_2 \in B' \backslash G /P$ be orbits satisfying
% 	$$\forall J \in \mathcal J, \, \rank_{J}X_1 \leqslant \rank_{J}X_2.$$
	
% 	{\color{red} Finir preuve (ii).}
% \end{proof}

\subsection[Case III')]{Case $\mathrm{III}')$} \label{cas 3'}

We let $\underline m = (m_1,\dots,m_l)$ be a partition of $n$ and
\begin{equation*}
	B' = \left\{\begin{pmatrix}
	a & \\
	& 1
\end{pmatrix} \mid a \in \T(n-1,\R)\right\} \subset GL(n,\R).
\end{equation*}
Following the $(GL(n,\R)/P_{\underline n})$-realization of $\D_{\underline n}$ from \Cref{flag varieties}, we will write any $D \in G/P_{\underline m}$ as
\begin{equation*}
	D = \left[\begin{tabular}{c|c|c}
	$A_1$ & $\cdots$ & $A_l$
\end{tabular}\right].
\end{equation*}
Moreover, we decompose each $A_i \in M_{n,m_i}(\R)$ as
\begin{equation*}
	A_i = \begin{bmatrix}
	B_{i,1} & \cdots & B_{i,m_i} \\
	a_{i,1} & \cdots & a_{i,m_i} 
\end{bmatrix}
\end{equation*} 
for some $a_{i,1},\dots,a_{i,m_i} \in \R$, $B_{i,1},\dots,B_{i,m_i} \in \M_{n-1,1}(\R)$.

In this setting, $\rank_{J,s}$ is $B'$-invariant if and only if
\begin{equation*}
	J \in \mathcal J \coloneq \Big\{\{i,\dots,n-1\} \mid 1 \leqslant i \leqslant n-1\Big\} \cup \Big\{\{i,\dots,n\} \mid 1 \leqslant i \leqslant n\Big\}.
\end{equation*}

\begin{Remark}
	Assume $P_{\underline m}$ is minimal, i.e. $P_{\underline m}$ is the Borel subgroup of $GL(n,\R)$. If $J \in \mathcal J$, define the map
	\begin{equation*}
		{\det}_J : \begin{array}{ccc}
		GL(n,\R) & \to & \R \\
		(a_{ij}) & \mapsto & \det((a_{ij})_{i \in J, n-|J|+1 \leqslant j \leqslant n})
	\end{array}.
	\end{equation*}
	The right $P_{\underline m}$-equivariance and left $B'$-equivariance of the maps $(\Phi_J)_{J \in \mathcal J}$ have been used in \cite{Frahm_2023}, \cite{DitlevsenFrahm} and \cite{DitlevsenLabriet} to both construct and study symmetry breaking operators for to the pair $(GL(n+1,\R),GL(n,\R))$
\end{Remark}

\begin{Proposition}{\label{theorem B proof exist}}
	Let $(e_1,\dots,e_{n-1})$ be the canonical basis of $\R^{n-1}$. Then, any $B'$-orbit on $\mathcal D_{\underline{m}} \simeq G/P_{\underline m}$ admits a representative given by
	\begin{equation*}
		A = \begin{bmatrix}
		A_1 & \cdots & A_{l}
	\end{bmatrix}
	\end{equation*}
	where each block $A_j$ is of the form
	\begin{equation}\label{eqn:normalform1proof}
		\begin{pmatrix}
			e_{i_{j,1}} & \dots & e_{i_{j,m_j}} \\
			0 & \dots & 0
		\end{pmatrix}
	\end{equation}
	except exactly one of them which is of the form
	\begin{equation}\label{eqn:normalform2proof}
		\begin{pmatrix}
			\displaystyle \sum_{s=1}^k e_{i_s} & e_{i_{j,2}} & \dots & e_{i_{j,m_j}} \\
			1 & 0 & \cdots & 0
		\end{pmatrix}
	\end{equation}
	for some $k \in \{0,\dots,l-1\}$. Moreover, it satisfies the following conditions
	\begin{itemize}[leftmargin=*, label = --]
		\item all $e_{i_{j,j'}}$ involved are distinct,
		\item if the block $A_{j_0}$ is of the form \eqref{eqn:normalform2proof}, then there exists an increasing sequence $j_0 < j_1 < \cdots < j_k$ such that $e_{i_s}$ appears in the block $A_{j_s}$ for all $s \in \{1,\dots,k\}$.
	\end{itemize}
\end{Proposition}
\begin{proof}
	Let $D \in \D_{\underline m}$ be given as $D = \left[\begin{tabular}{c|c|c}
		$A_1$ & $\cdots$ & $A_l$
	\end{tabular}\right]$ with each $A_i \in M_{n,m_i}(\R)$ being of the form
	\begin{equation*}
		A_i = \begin{bmatrix}
		B_{i,1} & \cdots & B_{i,m_i}\\
		a_{i,1} & \cdots & a_{i,m_i} 
	\end{bmatrix}.
	\end{equation*}
	Since the matrix $\begin{pmatrix}
		A_1 & \cdots & A_l
	\end{pmatrix} \in M_n(\R)$ has rank $n$, there exists $j \in \{1,\dots,A_l\}$, $i \in \{1,\dots,m_j\}$ such that $a_{j,i} \neq 0$ and we can set
	\begin{equation*}
		j_0 = \min\{j \in \{1,\dots,n\} \mid \exists i \in \{1,\dots,m_j\}, \, a_{j,i} \neq 0\}.
	\end{equation*}
	Using the right $P_{\underline m}$-action, we can make sure that $a_{j_0,1} = 1$ and that $a_{j,i} = 0$ if $(j,i) \neq (j_0,1)$.
	Applying the standard Gauss-Jordan elimination process (without row permutations) on the rectangular matrix obtained from $
		\begin{pmatrix}
		A_1 & \cdots & A_l
	\end{pmatrix}$ by removing the last row and the column $\,^t\begin{pmatrix}
		1 & B_{j_0,1}
	\end{pmatrix}$,
	we see that the $B'$-orbit of $D$ admits a representative with blocks $(A_j)_{j=1,\cdots,l}$ of the form
	\begin{equation*}
		\begin{pmatrix}
		e_{i_{j,1}} & \dots & e_{i_{j,m_j}} \\
		0 & \dots & 0
	\end{pmatrix}
	\end{equation*}
	except exactly one of them which is of the form
	\begin{equation*}
		\begin{pmatrix}
		v & e_{i_{j,2}} & \dots & e_{i_{j,m_j}} \\
		1 & 0 & \cdots & 0
	\end{pmatrix}
	\end{equation*}
	for some $v \in \R^{n-1}$, where all basis vectors $e_{i_{j,j'}}$ are necessarily distinct as the full matrix has rank $n$. Using the right $P_{\underline m}$-action, we can make sure that $v$ has zero $e_{i_{j,j'}}$-coordinate if $j \leqslant j_0$. Let $i_1<\cdots<i_k$ be the integers $p$ satisfying
	\begin{equation*}
		v \text{ has a non-zero } e_{p}\text{-coordinate } v_p, \hspace{1cm} \exists (j,j'), \, j_0 < j \text{ and } p = i_{j,j'}.
	\end{equation*}
	For each such integer $i_s$, we can first use the right $P_{\underline m}$-action to multiply each column $\,^t\begin{pmatrix}
		e_{i_s} & 0
	\end{pmatrix}$
	by the $e_{i_s}$-component $v_{i_s}$ and then use the left $B'$-action to multiply the $i_s$-th row by $v_{i_s}^{-1}$. Consequently $v$ transforms into $\sum_{s=1}^k e_{i_s}$	while the other columns are kept unchanged.
	The desired normal form is thus obtained.
\end{proof}

\begin{Proposition}
	In the setting of \Cref{theorem B proof exist}, the normal form is unique up to permutations of the $e_{j,j'}$ vectors in each block $A_j$.
\end{Proposition}
\begin{proof}
	Let us recall that, according to \Cref{projection equivariant}, the canonical projections $\Phi_{\underline m \to \underline m'}$ defined in \Cref{def projections} are $GL(n,\R)$-equivariant, hence $B'$-equivariant.

	Let $D \in \D_{\underline m}$ be a normal form as defined in \Cref{theorem B proof exist}. According to \Cref{normal form unique} for $(n_1,n_2) = (n-1,1)$, the indices $i_{1,j'}$ are entirely determined by the $B'$-orbit of $\Phi_{\underline m \to (m_1,n-m_1)}(D)$, hence by the $B'$-orbit of $D$ since $\Phi_{\underline m \to (m_1,n-m_1)}$ is $B'$-equivariant. We then see by induction that all the indices $i_{j,j'}$ are determined by the $B'$-orbit of $D$. 
	
	Similarly, $j_0$ is the smallest integer $p \in \{1,\dots,l\}$ satisfying 
	\begin{equation*}
		\Phi_{\underline m \to (m_1+\cdots+m_p,n-(m_1+\cdots+m_p))}(D)  = 1
	\end{equation*}
	and is thus entirely determined by the $B'$-orbit of $D$.
	
	If $k > 0$, the normal form of 
	\begin{equation*}
		\Phi_{\underline m \to (m_1+\cdots+m_{i_1},n-(m_1+\cdots+m_{i_1}))}(D)
	\end{equation*}
	consists of the blocks $\begin{pmatrix}
		A_1 & \cdots & A_{i_1}
	\end{pmatrix}$
	where we replace the column 
	\begin{equation*}
	\begin{pmatrix}
		\displaystyle \sum_{s=1}^ke_{i_s} \\
		1
	\end{pmatrix} \hspace{3mm} \text{ by } \hspace{3mm} \begin{pmatrix}
		e_{i_1} \\
		1
	\end{pmatrix}.
	\end{equation*}
	Thus, the integer $i_1$ is uniquely determined by the $B'$-orbit of $D$. Moreover, $j_1$ is the smallest integer $p \in \{j_0+1,\dots,l\}$ for which the normal form of 
	\begin{equation*}
		\Phi_{\underline m \to (m_1+\cdots+m_{p},n-(m_1+\cdots+m_{p}))}(D)
	\end{equation*}
	contains the column $\,^t\begin{pmatrix}
		e_{i_2} & 1
	\end{pmatrix}$ in place of the column $\,^t\begin{pmatrix}
		e_{i_1} & 1
	\end{pmatrix},$
	so that $j_1$ is uniquely determined by the $B'$-orbit of $D$. Again, we see by induction that both sequences $(i_1,\dots,i_k)$ and $(j_1,\dots,j_k)$ are uniquely determined by the orbit $B'$-orbit of $D$. The proof is now complete.
\end{proof}

The equivariance of the maps $\Phi_{\underline m \to \underline m'}$ ensures that they induce well-defined maps $B' \backslash G/P_{\underline m} \to B' \backslash G/P_{\underline m'}$. The proof of the previous proposition also implies the following corollary.
\begin{Corollary}
	Let $\mathcal Q_{B'}$ and $\mathcal Q_{P_{\underline m}}$ be the sets of maximal parabolic subgroups of $G$ containing $B'$ and $P_{\underline m}$ respectively. Then, the map
	\begin{equation*}
		\Psi : B'\backslash G/P_{\underline m} \to \displaystyle \prod_{\substack{L \in \mathcal Q_{B'} \\ H \in \mathcal Q_{P_{\underline m}}}} L\backslash G/H
	\end{equation*}
	is injective.
\end{Corollary}

%\begin{Corollary}
%	Let $(Q_i)_{i=1,\dots,\ell}$ be the maximal parabolic subgroups of $G$ containing $P$, let $B_1,B_2$ the two Borel subgroup of $G$ containing $P'$ and let $(Q_i')_{i=1,\dots,m}$ be the maximal parabolic subgroups of $G$ containing $P'$. If $(L_j)_j$ is either $(P')$, $(B_1,B_2)$ or $(Q_i')_{i=1,\dots,m}$ and $(H_i)_i$ is either $P$ or $(Q_i)_{i=1,\dots,\ell}$, the map
%	$$P'\backslash G/P \to  \prod_{i,j} L_j\backslash G/H_i$$
%	is injective and induces a poset isomorphism with its image.
%\end{Corollary}
%\begin{proof}
%	
%	
%\end{proof}
%
%

\subsection[Case I)]{Case $\mathrm{I})$} \label{cas 1}

We let $\underline n = (n_1,n_2,n_3)$ and $\underline m = (m_1,m_2)$ be compositions of $n$ satisfying $\min(n_1,n_2,n_3) = 1$ and $\min(m_1,m_2) \geqslant 2$.
Permuting $n_1,n_2,n_3$ if needed, we may assume that $n_1 = 1$ so that
\begin{equation*}
	B' = \left\{\begin{pmatrix}
			a & 0 & 0 \\
			0 & b & 0 \\
			0 & 0 & c
		\end{pmatrix} \mid a \in \T(n_1,\R) = \R^*, \, b \in \T(n_2,\R), \, c \in \T(n_3,\R)\right\} \subset G.
\end{equation*}
Following the $(GL(n,\R)/P_{\underline n})$-realization of $\D_{\underline n}$ from \Cref{flag varieties}, we will write any $D \in \D_{\underline m}$ as
\begin{equation*}
	D = \begin{bmatrix}
	a_1 & \cdots & a_{m_1} \\
	U_1 & \cdots & U_{m_1} \\
	V_1 & \cdots & V_{m_1}
\end{bmatrix}
\end{equation*}
with $a_1,\dots,a_{m_1} \in \R$, $U_1,\dots,U_{m_1} \in \R^{n_2}$ and $V_1,\dots,V_{m_1} \in \R^{n_3}$.

%If $W$ is such a $k$-plane, we may use the right $GL(k,\R)$ action to write
%$$W = \begin{bmatrix}
%	\varepsilon & 0 & \cdots & 0 \\
%	U_1 & U_2 & \cdots &  U_k \\
%	V_1 & V_2 & \cdots & V_k
%\end{bmatrix}$$
%with $\varepsilon \in \{\pm 1\}$.
%Then, using the previous case on the "submatrix"
%$$\begin{pmatrix}
%	U_2 & \cdots & U_k \\
%	V_2 & \cdots & V_k
%\end{pmatrix},$$
%we obtain
%$$W \sim \begin{bmatrix}
%	\varepsilon & 0 & \cdots & 0 & 0 & \cdots & 0 & 0 & \cdots & 0\\
%	U_1 & 0 & \cdots & 0  & e_{j_{\alpha}+1} & \cdots & e_{j_{t}} & e_{j_{t+1}} & \cdots & e_{j_{k}} \\
%	V_1 & f_{i_1} & \cdots & f_{i_{\alpha}}  & f_{i_{\alpha}+1} & \cdots & f_{i_t} & 0 & \cdots & 0
%\end{bmatrix}$$
%and again, using the right $GL(k,\R)$-action, we may assume that $U_1$ has zero $e_{j_{\alpha}+1}, \dots, e_{j_{k}}$ components and that $V_1$ has zero $f_{i_1},\dots,f_{i_t}$ components. Letting
%$$j_0 = \max\{1 \leqslant j \leqslant r \mid U_1 \text{ has non-zero } e_j \text{ component}\}$$
%$$i_0 = \max\{1 \leqslant i \leqslant r \mid V_1 \text{ has non-zero } f_i \text{ component}\},$$
%we get from the left $P'$-action

\begin{Proposition}{\label{orbits I}}
	Any flag $D \in \D_{\underline m}$ belongs to the orbit of a flag of the form
	\begin{equation*}
		\begin{bmatrix}
		a_1 & 0 & \cdots & 0 & 0 & \cdots & 0 & 0 & \cdots & 0\\
		U_1 & 0 & \cdots & 0  & e_{j_{r}+1} & \cdots & e_{j_{s}} & e_{j_{s+1}} & \cdots & e_{j_{m_1}} \\
		V_1 & f_{i_1} & \cdots & f_{i_{r}}  & f_{i_{r}+1} & \cdots & f_{i_s} & 0 & \cdots & 0
	\end{bmatrix}
	\end{equation*}
	with
	\begin{equation*}
		\begin{pmatrix}
		a_1 \\
		U_1 \\
		V_1
	\end{pmatrix} = \begin{pmatrix}
	1 \\
	0 \\
	0
	\end{pmatrix}, \begin{pmatrix}
	1 \\
	e_{j_0} \\
	0
	\end{pmatrix}, \begin{pmatrix}
	1 \\
	0 \\
	f_{i_0}
	\end{pmatrix}, \begin{pmatrix}
	1 \\
	e_{j_0} \\
	f_{i_0}
	\end{pmatrix}, \begin{pmatrix}
	0 \\
	e_{j_0} \\
	0
	\end{pmatrix}, \begin{pmatrix}
	0 \\
	0 \\
	f_{i_0}
	\end{pmatrix}, \begin{pmatrix}
	0 \\
	e_{j_0} \\
	f_{i_0}
	\end{pmatrix}.
	\end{equation*}
\end{Proposition}

\begin{Theorem}
	The flags described in \Cref{orbits I} have different invariant ranks. In particular, the normal form of \Cref{orbits I} is unique and the map
	\begin{equation*}
		\Psi : B'\backslash G/P_{\underline m} \to \displaystyle \prod_{\substack{L \in \mathcal Q_{B'} \\ H \in \mathcal Q_{P_{\underline m}}}} L\backslash G/H
	\end{equation*}
	is injective.
\end{Theorem}
%$$W \sim \begin{bmatrix}
%	\varepsilon & 0 & \cdots & 0 & 0 & \cdots & 0 & 0 & \cdots & 0\\
%	\varepsilon_1 e_{j_0} & 0 & \cdots & 0  & e_{j_{\alpha}+1} & \cdots & e_{j_{t}} & e_{j_{t+1}} & \cdots & e_{j_{k}} \\
%	\varepsilon_2 f_{i_0} & f_{i_1} & \cdots & f_{i_{\alpha}}  & f_{i_{\alpha}+1} & \cdots & f_{i_t} & 0 & \cdots & 0
%\end{bmatrix}$$

%Using the exact same procedure as in the previous case, we get
%\begin{Proposition}
%	This normal form is unique up to permutation of columns.
%\end{Proposition}

\begin{Corollary}
	Closed $B'$-orbits are given by the following normal forms
	\begin{equation*}
		\begin{bmatrix}
		1 & 0 & \cdots & 0 & 0 & \cdots & 0\\
		0 & 0 & \cdots & 0  & e_{1} & \cdots & e_{k-t-1} \\
		0 & f_{1} & \cdots & f_{t}  & 0 & \cdots & 0
	\end{bmatrix}, \begin{bmatrix}
		0 & 0 & \cdots & 0 & 0 & \cdots & 0\\
		0 & 0 & \cdots & 0  & e_{1} & \cdots & e_{k-t} \\
		f_{1} & f_{2} & \cdots & f_{t}  & 0 & \cdots & 0
	\end{bmatrix}.
	\end{equation*}
\end{Corollary}
%\begin{Remark}
%	Note that in this context, all closed orbits are fixed points.
%\end{Remark}

\subsection[Case II)]{Case $\mathrm{II})$} \label{cas 2}

We let $\underline n = (n_1,n_2,n_3)$, $\underline m = (m_1,m_2)$ be compositions of $n$ satisfying $\min(n_1,n_2,n_3) \geqslant 2$ and $\min(m_1,m_2) = 2$. Consider
\begin{equation*}
	B' = \left\{\begin{pmatrix}
	a & 0 & 0 \\
	0 & b & 0 \\
	0 & 0 & c
\end{pmatrix} \mid a \in \T(n_1,\R), \, b \in \T(n_2,\R), \, c \in \T(n_3,\R)\right\}.
\end{equation*}
We distinguish between two subcases: $(m_1,m_2) = (2,n-2)$ and $(m_1,m_2) = (n-2,2)$. Moreover, for each $1 \leqslant i \leqslant 3$, we let $e_{1,i},\dots,e_{n_i,i}$ be the canonical basis of $\R^{n_i}$.

\subsubsection*{Subcase $(m_1,m_2) = (2,n-2)$}

Following the $(GL(n,\R)/P_{\underline n})$-realization of $\D_{\underline n}$ from \Cref{flag varieties}, we will write any $D \in \D_{(2,n-2)}$ as
\begin{equation*}
	D = \begin{bmatrix}
	U_1 & V_1 \\
	U_2 & V_2 \\
	U_3 & V_3
\end{bmatrix}
\end{equation*}
with $U_1,V_1 \in \R^{n_1}$, $U_2,V_2 \in \R^{n_2}$ and $U_3,V_3 \in \R^{n_3}$.

\begin{Proposition}{\label{orbits II1}}
	Any flag $D \in \D_{(2,n-2)}$ belongs to the $B'$-orbit of a flag of the form
	\begin{equation*}
		D = \begin{bmatrix}
	U_1 & V_1 \\
	U_2 & V_2 \\
	U_3 & V_3
\end{bmatrix}
	\end{equation*}
	where:
	\begin{itemize}[label = --]
		\item each $U_i$ and each $V_i$ is either $0 \in \R^{n_i}$ or a basis vector $e_{j_i,i} \in \R^{n_i}$,
		\item whenever both $U_i$ and $V_i$ are non-zero then $U_i \neq V_i$,
		\item at least one $U_i$ and one $V_j$ are non-zero.
	\end{itemize}
\end{Proposition}

\begin{Theorem}
	The flags described in \Cref{orbits II1} have different invariant ranks. In particular, the normal form of \Cref{orbits II1} is unique and the map
	\begin{equation*}
		\Psi : B'\backslash G/P_{\underline m} \to \displaystyle \prod_{\substack{L \in \mathcal Q_{B'} \\ H \in \mathcal Q_{P_{\underline m}}}} L\backslash G/H
	\end{equation*}
	is injective.
\end{Theorem}

\begin{Corollary}
	Closed orbits are given by the normal forms
	\begin{equation*}
		\begin{bmatrix}
	U_1 & V_1 \\
	U_2 & V_2 \\
	U_3 & V_3
\end{bmatrix}
	\end{equation*}
for which exactly one $U_i$ is non zero and equal to $e_{1,i}$ and one $V_j$ is non-zero and equal to $e_{1,j}$.
\end{Corollary}
% \begin{Remark}
% 	Note that in this context, all closed orbits are fixed points.
% \end{Remark}

\subsubsection*{Subcase $(m_1,m_2) = (n-2,2)$}{\label{dualité}}

Let $\R^n$ be equipped with its usual inner product $\langle, \rangle$. We will write $u^*$ and $V^\perp$ respectively for the Euclidean adjoint of a linear map $u : \R^n \to \R^n$ and for the orthogonal complement of a linear subspace $V \subset \R^n$.
In this setting, the map
\begin{equation}{\label{homeo dual}}
	\Gamma : \begin{array}{ccc}
		\D_{(n-2,2)} & \to & \D_{(2,n-2)} \\
		(V_1 \subset \R^n) & \mapsto & (V_1^\perp \subset \R^n)
	\end{array}
\end{equation}
is a homeomorphism. 
\begin{Lemma}{\label{lemme adjoint}}
	If $u \in GL(\R^n)$ and $V \subset \R^n$ is a linear subspace, one has
	\begin{equation*}
		u(V)^\perp = (u^*)^{-1}(V^\perp)
	\end{equation*}
	where $u^*$ is the Euclidean adjoint map of $u$.
\end{Lemma}
\begin{proof}
	If $u \in GL(\R^n)$, $x \in \R^n$ and $z \in V$, one has
	\begin{align*}
		x \in u(V)^\perp &\iff \forall y \in V, \, \langle x,u(y) \rangle = 0 \\
		&\iff \forall y \in V, \, \langle u^*(x),y \rangle = 0 \\
		&\iff u^*(x) \in V^\perp \\
		&\iff x \in (u^*)^{-1}(V^\perp).
	\end{align*}
\end{proof}
\begin{Lemma}
	Under the homeomorphism \eqref{homeo dual} the action of $B'$ on $\D_{(2,n-2)}$ is given by the left-multiplication by inverse transpose.
\end{Lemma}
\begin{proof}
	According to \Cref{lemme adjoint}, the following diagram is commutative for any $u \in GL(\R^n)$:
	
	\begin{center}
		\begin{tikzcd}[column sep={1.5cm},row sep={1.5cm}]
			\D_{(n-2,2)} \arrow[r, "\Gamma"] \arrow[d,"u", swap] & \D_{(2,n-2)} \arrow[d, "(u^*)^{-1}"]\\
			\D_{(n-2,2)} \arrow[r,"\Gamma",swap] & \D_{(2,n-2)}
		\end{tikzcd}
	\end{center}
	Since the matrix of the inverse adjoint map of a linear map is the inverse transpose of the matrix of the linear map, the lemma follows.
\end{proof}

\begin{Corollary}
	Under the heomeomorphism \eqref{homeo dual}, $B'$-orbits on $\D_{(n-2,2)}$ correspond to the $\overline{B'}$-orbits on $\D_{(2,n-2)}$ where 
	$\overline{B'} = \{\,^tg \mid g \in B'\}.$
\end{Corollary}

\begin{Proposition}{\label{orbits II2}}
	Any flag $D \in \D_{(n-2,2)}$ belongs to the $B'$-orbit of a flag of the form
	\begin{equation*}
		\Gamma^{-1}\left(\begin{bmatrix}
	U_1 & V_1 \\
	U_2 & V_2 \\
	U_3 & V_3
\end{bmatrix}\right)
	\end{equation*}
	where:
	\begin{itemize}[label = --]
		\item each $U_i$ and each $V_i$ is either $0 \in \R^{n_i}$ or a basis vector $e_{j_i,i} \in \R^{n_i}$,
		\item whenever both $U_i$ and $V_i$ are non-zero then $U_i \neq V_i$,
		\item at least one $U_i$ and one $V_j$ are non-zero.
	\end{itemize}
\end{Proposition}

\begin{Theorem}
	The flags described in \Cref{orbits II2} have different invariant ranks. In particular, the normal form of \Cref{orbits II2} is unique and the map
	\begin{equation*}
		\Psi : B'\backslash G/P_{\underline m} \to \displaystyle \prod_{L \in \mathcal Q_{B'}, H \in \mathcal Q_{P_{\underline m}}} L\backslash G/H
	\end{equation*}
	is injective.
\end{Theorem}

\begin{Corollary}
	Closed $B'$-orbits are given by the normal forms
	\begin{equation*}
		\Gamma^{-1}\left(\begin{bmatrix}
	U_1 & V_1 \\
	U_2 & V_2 \\
	U_3 & V_3
\end{bmatrix}\right)
	\end{equation*}
for which exactly one $U_i$ is non zero and equal to $e_{n_i,i}$ and one $V_j$ is non-zero and equal to $e_{n_j,j}$.
\end{Corollary}

\subsection[Case III)]{Case $\mathrm{III})$} \label{cas 3}

We let $\underline n = (n_1,\dots,n_k)$ and $\underline m = (m_1,m_2)$ be compositions of $n$ satisfying $\min(m_1,m_2) = 1$ and we consider
\begin{equation*}
	B' = \left\{\begin{pmatrix}
	a_1 &  &  \\
	& \ddots & \\
	&  & a_k
\end{pmatrix} \mid \forall i \in \{1,\dots,k\}, \, a_i \in \T(n_i,\R),\right\} \subset GL(n,\R).
\end{equation*}
We distinguish between two subcases: $(m_1,m_2) = (1,n-1)$ and $(m_1,m_2) = (n-1,1)$.

\subsubsection*{Subcase $(m_1,m_2) = (1,n-1)$}

Following the $(GL(n,\R)/P_{\underline n})$-realization of $\D_{\underline n}$ from \Cref{flag varieties}, we will write any $D \in \D_{\underline m}$ as
\begin{equation*}
	D = \,^t\begin{bmatrix}
	U_1 & \cdots & U_k
\end{bmatrix}
\end{equation*}
with $U_1 \in M_{1,n_1}(\R),\dots, U_k \in M_{1,n_k}(\R)$. Moreover, for each $1 \leqslant i \leqslant k$, we let $e_{1,i},\dots,e_{n_i,i}$ be the canonical basis of $\R^{n_i}$.

\begin{Proposition}{\label{orbits III1}}
	Any flag $D \in \D_{(1,n-1)}$ belongs to the $B'$-orbit of a flag of the form
	\begin{equation*}
		\,^t\begin{bmatrix}
		U_1 & \cdots & U_k
	\end{bmatrix}
	\end{equation*}
	where each $U_i$ is either $(0,\dots,0) \in M_{1,n_i}(\R)$ or a basis vector $e_{j_i,i} \in \R^{n_i}$ with at least one non-zero $U_i$.
\end{Proposition}

\begin{Theorem}
	The flags described in \Cref{orbits III1} have different invariant ranks. In particular, the normal form of \Cref{orbits III1} is unique and the map
	\begin{equation*}
		\Psi : B'\backslash G/P_{\underline m} \to \displaystyle \prod_{\substack{L \in \mathcal Q_{B'} \\ H \in \mathcal Q_{P_{\underline m}}}} L\backslash G/H
	\end{equation*}
	is injective.
\end{Theorem}

\begin{Corollary}
	Closed orbits are given by the normal forms
	\begin{equation*}
		\,^t\begin{bmatrix}
		U_1 & \cdots & U_k
	\end{bmatrix}
	\end{equation*}
	for which a unique $U_i$ is non-zero and equal to $e_{1,i}$.
\end{Corollary}

\subsubsection*{Subcase $(m_1,m_2) = (n-1,1)$}

Following the same duality trick used in \Cref{dualité}, we identify the $B'$-orbits on $\D_{(n-1,1)}$ with the $\overline{B'}$-orbits on $\D_{(1,n-1)}$ using the map
\begin{equation*}
	\Gamma : \begin{array}{ccc}
		\D_{(n-1,1)} & \to & \D_{(1,n-1)} \\
		(V_1 \subset \R^n) & \mapsto & (V_1^\perp \subset \R^n)
	\end{array}
\end{equation*}

\begin{Proposition}{\label{orbits III2}}
	Any flag $D \in \D_{(n-1,1)}$ belongs to the $B'$-orbit of a flag of the form
	\begin{equation*}
		\Gamma^{-1}\left(\,^t\begin{bmatrix}
		U_1 & \cdots & U_k
	\end{bmatrix}\right)
	\end{equation*}
	where each $U_i$ is either $(0,\dots,0) \in M_{1,n_i}(\R)$ or a basis vector $e_{j,i} \in \R^{n_i}$ with at least one non-zero $U_i$.
\end{Proposition}

\begin{Theorem}
	The flags described in \Cref{orbits III2} have different invariant ranks. In particular, the normal form of \Cref{orbits III2} is unique and the map
	\begin{equation*}
		\Psi : B'\backslash G/P_{\underline m} \to \displaystyle \prod_{\substack{L \in \mathcal Q_{B'} \\ H \in \mathcal Q_{P_{\underline m}}}} L\backslash G/H
	\end{equation*}
	is injective.
\end{Theorem}

\begin{Corollary}
	Closed orbits are given by the normal forms $\Gamma^{-1}\left(\,^t\begin{bmatrix}
		U_1 & \cdots & U_k
	\end{bmatrix}\right)$
	for which a unique $U_i$ is non-zero and equal to $e_{n_i,i}$.
\end{Corollary}

\subsection[Case I')]{Case $\mathrm{I}')$} \label{cas 1'}

We let $\underline n = (n_1,n_2)$ and $\underline m = (m_1,m_2,m_3)$ be compositions of $n$ satisfying $\min(n_1,n_2) \geqslant 2$ and $\min(m_1,m_2,m_3) = 1$ and we consider
\begin{equation*}
	B' = \left\{\begin{pmatrix}
	a_1 & \\
	& a_2
\end{pmatrix} \mid a_1 \in \T(n_1,\R), a_2 \in \T(n_2,\R)\right\} \subset GL(n,\R).
\end{equation*}
We split between three subcases:
\begin{equation*}
	m_1 = 1 \text{ and } m_2 \neq 1, \hspace{5mm} m_2 = 1, \hspace{5mm} m_3 = 1 \text{ and } m_2 \neq 1.
\end{equation*}

\subsubsection*{Subcase $m_1 = 1$, $m_2 \neq 1$}

Following the $(GL(n,\R)/P_{\underline n})$-realization of $\D_{\underline n}$ from \Cref{flag varieties}, we will write any $D \in \D_{\underline m}$ as
\begin{equation*}
	D = \left[\begin{tabular}{c|ccc}
	$U$ & $U_1$ & $\cdots$ & $U_{m_2}$  \\
	$V$ & $V_1$ & $\cdots$ & $V_{m_2}$ 
\end{tabular}\right]
\end{equation*}
with $U,U_1,\dots, U_{m_2} \in \R^{n_1}$, $V,V_1,\dots,V_{m_2} \in \R^{n_2}$.

\begin{Theorem}{\label{contre exemple}}
	If $n \geqslant 5$, the map $\Psi : B'\backslash G/P_{\underline m} \to \displaystyle \prod_{\substack{L \in \mathcal Q_{B'} \\ H \in \mathcal Q_{P_{\underline m}}}} L\backslash G/H$ is not injective.
\end{Theorem}
\begin{proof}
	Let us first treat the case of $\underline n = (3,2)$ and $\underline m = (1,2,2)$. Consider the flags
	\begin{equation*}
		D_1 = \left[\begin{tabular}{c|cc}
	0 & 1 & 0 \\
	0 & 0 & 1 \\
	1 & 0 & 0 \\
	0 & 0 & 1 \\
	1 & 1 & 0 
\end{tabular}\right], \hspace{3mm} D_2 = \left[\begin{tabular}{c|cc}
	0 & 1 & 0 \\
	0 & 0 & 1 \\
	1 & 0 & 0 \\
	0 & 1 & 1 \\
	1 & 1 & 0 
\end{tabular}\right].
	\end{equation*}
Note that
\begin{equation*}
	\Phi_{\underline m \to (1,4)}(D_1) = \left[\begin{tabular}{c}
	0\\
	0\\
	1\\
	0\\
	1
\end{tabular}\right], \hspace{5mm} \Phi_{\underline m \to (1,4)}(D_2) = \left[\begin{tabular}{c}
	0\\
	0\\
	1\\
	0\\
	1
\end{tabular}\right],
\end{equation*}
\begin{equation*}
	\Phi_{\underline m \to (3,2)}(D_1) = \left[\begin{tabular}{ccc}
	-1 & 1 & 0 \\
	0 & 0 & 1 \\
	1 & 0 & 0 \\
	0 & 0 & 1 \\
	0 & 1 & 0 
\end{tabular}\right] \sim \left[\begin{tabular}{ccc}
	0 & 1 & 0 \\
	0 & 0 & 1 \\
	1 & 0 & 0 \\
	0 & 0 & 1 \\
	0 & 1 & 0 
\end{tabular}\right],
\end{equation*}
\begin{equation*}
	\Phi_{\underline m \to (3,2)}(D_2) = \left[\begin{tabular}{ccc}
	-1 & 1 & 0 \\
	0 & 0 & 1 \\
	1 & 0 & 0 \\
	-1 & 1 & 1 \\
	0 & 1 & 0 
\end{tabular}\right] \sim \left[\begin{tabular}{ccc}
	0 & 1 & 0 \\
	0 & 0 & 1 \\
	1 & 0 & 0 \\
	1 & 0 & 1 \\
	0 & 1 & 0 
\end{tabular}\right] = \left[\begin{tabular}{ccc}
	0 & 1 & 0 \\
	-1 & 0 & 1 \\
	1 & 0 & 0 \\
	0 & 0 & 1 \\
	0 & 1 & 0 
\end{tabular}\right] \sim \left[\begin{tabular}{ccc}
	0 & 1 & 0 \\
	0 & 0 & 1 \\
	1 & 0 & 0 \\
	0 & 0 & 1 \\
	0 & 1 & 0 
\end{tabular}\right]
\end{equation*}
so that $D_1$ and $D_2$ have the same $B'$-invariant ranks.
On the other hand, if we suppose that $D_1$ and $D_2$ are in the same $B'$-orbit then there exists 
\begin{equation*}
	b' = \begin{pmatrix}
	\alpha & \beta & \gamma & 0 & 0 \\
	0 & \delta & \varepsilon & 0 & 0 \\
	0 & 0 & \zeta & 0 & 0 \\
	0 & 0 & 0 & \eta & \theta \\
	0 & 0 & 0 & 0 & \iota
\end{pmatrix} \in B', \hspace{5mm} p = \begin{pmatrix}
	\kappa & \lambda & \mu \\
	0 & \nu & \xi \\
	0 & \rho & \sigma \\
\end{pmatrix} \in P_{\underline m},
\end{equation*}
such that
\begin{equation*}
	b' \begin{pmatrix}
	0 & 1 & 0 \\
	0 & 0 & 1 \\
	1 & 0 & 0 \\
	1 & 0 & 1 \\
	0 & 1 & 0
\end{pmatrix}p = \begin{pmatrix}
	0 & 1 & 0 \\
	0 & 0 & 1 \\
	1 & 0 & 0 \\
	1 & 1 & 1 \\
	0 & 1 & 0
\end{pmatrix},
\end{equation*}
namely
\begin{equation*}
	\begin{pmatrix}
	\ast & \ast & \ast \\
	\kappa \varepsilon & \lambda \varepsilon + \rho \delta & \ast \\
	\ast & \ast & \ast \\
	\kappa \theta & \lambda \theta + \nu \theta + \rho \eta & \ast \\
	\ast & \ast & \ast
\end{pmatrix} = \begin{pmatrix}
	0 & 1 & 0 \\
	0 & 0 & 1 \\
	1 & 0 & 0 \\
	1 & 1 & 1 \\
	0 & 1 & 0
\end{pmatrix}.
\end{equation*}
In particular, one has
\begin{equation*}
	\begin{cases}
	\kappa \theta = 0 \\
	(\lambda+\nu)\theta + \rho \eta = 1 \\
	\kappa \varepsilon = 0 \\
	\lambda \varepsilon + \rho \delta = 0
\end{cases}.
\end{equation*}
Since $\kappa \neq 0$, this implies that $\theta = \varepsilon = 0$ so that $\rho \delta = 0$. In turn, $\delta \neq 0$ so that $\rho = 0$ what is impossible in view of the second equation. Thus, we conclude that $D_1$ and $D_2$ belong to different $B'$-orbits. The theorem then follows for general $\underline n$ and $\underline m$ from observing that $\max(n_1,n_2) \geqslant 3$ whenever $n \geqslant 5$.
\end{proof}

\subsubsection*{Subcase $m_2 = 1$}

Following the $(GL(n,\R)/P_{\underline n})$-realization of $\D_{\underline n}$ from \Cref{flag varieties}, we will write any $D \in \D_{\underline m}$ as
\begin{equation*}
	D = \left[\begin{tabular}{ccc|c}
	$U_1$ & $\cdots$ & $U_{m_1}$ & $U$ \\
	$V_1$ & $\cdots$ & $V_{m_1}$ & $V$
\end{tabular}\right]
\end{equation*}
with $U,U_1,\dots, U_{m_1} \in \R^{n_1}$, $V,V_1,\dots,V_{m_1} \in \R^{n_2}$.

\begin{Proposition}{\label{orbits I3'}}
	Any flag $D \in \D_{\underline m}$ belongs to the $B'$-orbit of a flag of the form
	\begin{equation*}
		\left[\begin{tabular}{ccccccccc|c}
		$0$ & $\cdots$ & $0$  & $e_{j_{r+1}}$ & $\cdots$ & $e_{j_{s}}$ & $e_{j_{s+1}}$ & $\cdots$ & $e_{j_{m_1}}$ & $U$\\
		$f_{i_1}$ & $\cdots$ & $f_{i_{r}}$  & $f_{i_{r+1}}$ & $\cdots$ & $f_{i_{s}}$ & $0$ & $\cdots$ & $0$ & $V$
	\end{tabular}\right]
	\end{equation*}
	with either
	\begin{itemize}[label = --]
		\item $\begin{pmatrix}
			U \\
			V
		\end{pmatrix} = \begin{pmatrix}
			e_{j_0} \\
			0
		\end{pmatrix}, \begin{pmatrix}
			0 \\
			f_{i_0}
		\end{pmatrix}, \begin{pmatrix}
			e_{j_0} \\
			f_{i_0}
		\end{pmatrix}$,
		\item $\begin{pmatrix}
			U \\
			V
		\end{pmatrix} = \begin{pmatrix}
			0 \\
			\displaystyle \sum_{p=1}^q f_{i_{k_p}}
		\end{pmatrix}$ or $\begin{pmatrix}
			\displaystyle \sum_{p=1}^q e_{j_{k_p}} \\
			0
		\end{pmatrix}$ with some integers $k_1,\dots,k_q \in \{r+1,\dots,s\}$ satisfying
		\begin{equation*}
			i_{k_1} < \cdots < i_{k_q}, \hspace{3mm} j_{k_1} > \cdots > j_{k_q},
		\end{equation*}
		\item $\begin{pmatrix}
			U \\
			V
		\end{pmatrix} = \begin{pmatrix}
			0 \\
			\displaystyle f_{i_0}+\sum_{p=1}^q f_{i_{k_p}}
		\end{pmatrix}$ or $\begin{pmatrix}
			\displaystyle e_{j_0}+\sum_{p=1}^q e_{j_{k_p}} \\
			0
		\end{pmatrix}$ with some integers $k_1,\dots,k_q \in \{r+1,\dots,s\}$ satisfying $i_{k_1} < \cdots < i_{k_q}, \hspace{3mm} j_{k_1} > \cdots > j_{k_q}$ and $i_0 < i_{k_1}$ or $j_0 < j_{k_q}$.
		\item $\begin{pmatrix}
			U \\
			V
		\end{pmatrix} = \begin{pmatrix}
			e_{j_0} \\
			\displaystyle f_{i_0}+\sum_{p=1}^q f_{i_{k_p}}
		\end{pmatrix}$ with some integers $k_1,\dots,k_q \in \{r+1,\dots,s\}$ satisfying
		\begin{equation*}
			i_0 < i_{k_1} < \cdots < i_{k_q}, \hspace{3mm} j_{k_1} > \cdots > j_{k_q} > j_0,
		\end{equation*}
	\end{itemize}
\end{Proposition}

\begin{Theorem}
	The flags described in \Cref{orbits I3'} have different invariant ranks. In particular, the normal form of \Cref{orbits I3'} is unique and the map
	\begin{equation*}
		\Psi : B'\backslash G/P_{\underline m} \to \displaystyle \prod_{\substack{L \in \mathcal Q_{B'} \\ H \in \mathcal Q_{P_{\underline m}}}} L\backslash G/H
	\end{equation*}
	is injective.
\end{Theorem}

\begin{Corollary}
	Closed orbits are given by the normal forms
	\begin{equation*}
		\left[\begin{tabular}{cccccc|c}
		$0$ & $\cdots$ & $0$ & $e_{1}$ & $\cdots$ & $e_{m_1-r}$ & $e_{m_1-r+1}$ \\
		$f_1$ & $\cdots$ & $f_r$ & $0$ & $\cdots$ & $0$ & $0$
	\end{tabular}\right], \hspace{3mm} \left[\begin{tabular}{cccccc|c}
		$0$ & $\cdots$ & $0$ & $e_{1}$ & $\cdots$ & $e_{m_1-r}$ & $0$ \\
		$f_1$ & $\cdots$ & $f_r$ & $0$ & $\cdots$ & $0$ & $f_{r+1}$
	\end{tabular}\right].
	\end{equation*}
\end{Corollary}

\subsubsection*{Subcase $m_3 = 1$ and $m_2 \neq 1$} \label{contre exemple 2}

Following the same duality trick used in \Cref{dualité}, we identify the $B'$-orbits on $\D_{(m_1,m_2,m_3)}$ with the $\overline{B'}$-orbits on $\D_{(m_3,m_2,m_1)}$ using the map 
\begin{equation*}
	\Gamma : \begin{array}{ccc}
		\D_{(m_1,m_2,m_3)} & \to & \D_{(m_3,m_2,m_1)} \\
		(V_1 \subset \R^n) & \mapsto & (V_1^\perp \subset \R^n)
	\end{array}
\end{equation*}

\begin{Theorem}
	If $n \geqslant 5$, the map $\Psi : B'\backslash G/P_{\underline m} \to \displaystyle \prod_{\substack{L \in \mathcal Q_{B'} \\ H \in \mathcal Q_{P_{\underline m}}}} L\backslash G/H$ is not injective.
\end{Theorem}
\begin{proof}
	The proof is the same as for \Cref{contre exemple} except we use the flags
	\begin{equation*}
		D_1 = \Gamma^{-1}\left(\left[\begin{tabular}{c|cc}
	1 & 0 & 0 \\
	0 & 0 & 1 \\
	0 & 1 & 0 \\
	1 & 1 & 0 \\
	0 & 0 & 1 
	\end{tabular}\right]\right), \hspace{3mm} D_2 = \Gamma^{-1}\left(\left[\begin{tabular}{c|cc}
		1 & 0 & 0 \\
		0 & 0 & 1 \\
		0 & 1 & 0 \\
		1 & 1 & 0 \\
		0 & 1 & 1 
	\end{tabular}\right]\right).
	\end{equation*}
\end{proof}

\subsection[Case II')]{Case $\mathrm{II}')$} \label{cas 2'}

We let $\underline n = (n_1,n_2)$ and $\underline m = (m_1,m_2,m_3)$ be compositions of $n$ satisfying $\min(n_1,n_2) = 2$ and $\min(m_1,m_2,m_3) \geqslant 2$. Permuting $n_1$ and $n_2$ if needed, we may assume that $\underline n = (n-2,2)$. Consider
\begin{equation*}
	B' = \left\{\begin{pmatrix}
	a_1 & \\
	& a_2
\end{pmatrix} \mid a_1 \in \T(n-2,\R), a_2 \in \T(2,\R)\right\} \subset GL(n,\R).
\end{equation*}
Following the $(GL(n,\R)/P_{\underline n})$-realization of $\D_{\underline n}$ from \Cref{flag varieties}, we will write any $D \in \D_{\underline m}$ as
\begin{equation*}
	D = \left[\begin{tabular}{ccc|ccc}
	$U_1$ & $\cdots$ & $U_{m_1}$ & $U_{m_1+1}$ & $\cdots$ & $U_{m_1+m_2}$ \\
	$V_1$ & $\cdots$ & $V_{m_1}$ & $V_{m_1+1}$ & $\cdots$ & $V_{m_1+m_2}$
\end{tabular}\right]
\end{equation*}
with $U_1,\dots, U_{m_1+m_2} \in \R^{n-2}$, $V_1,\dots,V_{m_1+m_2} \in \R^{2}$.

\begin{Theorem}
	The map $\Psi : B'\backslash G/P_{\underline m} \to \displaystyle \prod_{\substack{L \in \mathcal Q_{B'} \\ H \in \mathcal Q_{P_{\underline m}}}} L\backslash G/H$ is not injective.
\end{Theorem}
\begin{proof}
	Since $l = 3$ and $M \geqslant 2$, one has $n \geqslant 6$.
	The proof is the same as for \Cref{contre exemple} but with the flags
	\begin{equation*}
		D_1 = \left[\begin{tabular}{cc|cc}
	1 & 0 & 0 & 0 \\
	0 & 0 & 1 & 0 \\
	0 & 0 & 0 & 1 \\
	0 & 1 & 0 & 0 \\
	0 & 0 & 0 & 1 \\
	0 & 1 & 1 & 0 
\end{tabular}\right], \hspace{3mm} D_2 = \left[\begin{tabular}{cc|cc}
	1 & 0 & 0 & 0 \\
	0 & 0 & 1 & 0 \\
	0 & 0 & 0 & 1 \\
	0 & 1 & 0 & 0 \\
	0 & 0 & 1 & 1 \\
	0 & 1 & 1 & 0 
\end{tabular}\right].
	\end{equation*}
\end{proof}

% \begin{Corollary}
% 	Closed orbits are given by the normal forms
% 	$$\left[\begin{tabular}{ccc|ccc}
% 		$e_1$ & $\cdots$ & $e_{m_1}$ & $e_{m_1+1}$ & $\cdots$ & $e_{m_1+m_2}$ \\
% 		$0$ & $\cdots$ & $0$ & $0$ & $\cdots$ & $0$
% \end{tabular}\right], \hspace{3mm} \left[\begin{tabular}{cccc|ccc}
% 	$0$ & $e_1$ & $e_{m_1-1}$ & $e_{m_1}$ & $\cdots$ & $e_{m_1+m_2-1}$ \\
% 	$f_1$ & $0$ & $\cdots$ & $0$ & $\cdots$ & $0$
% \end{tabular}\right],$$
% $$\left[\begin{tabular}{ccc|cccc}
% 	$e_1$ & $\cdots$ & $e_{m_1}$ & $0$ & $e_{m_1+1}$ & $\cdots$ & $e_{m_1+m_2-1}$ \\
% 	$0$ & $\cdots$ & $0$ & $f_1$ & $0$ & $\cdots$ & $0$
% \end{tabular}\right], \hspace{3mm} \left[\begin{tabular}{ccccc|ccc}
% 	$0$ & $0$ & $e_1$ & $\cdots$ & $e_{m_1-2}$ & $e_{m_1-1}$ & $\cdots$ & $e_{m_1+m_2-2}$ \\
% 	$f_1$ & $f_2$ & $0$ & $\cdots$ & $0$ & $0$ & $\cdots$ & $0$
% \end{tabular}\right],$$
% $$\left[\begin{tabular}{cccc|cccc}
% 	$0$ & $e_1$ & $\cdots$ & $e_{m_1-1}$ & $0$ & $e_{m_1}$ & $\cdots$ & $e_{m_1+m_2-1}$\\
% 	$f_1$ & $0$ & $\cdots$ & $0$ & $f_2$ & $0$ & $\cdots$ & $0$
% \end{tabular}\right], \hspace{3mm} \left[\begin{tabular}{ccc|ccccc}
% 	$e_1$ & $\cdots$ & $e_{m_1}$ & $0$ & $0$ & $e_{m_1+1}$ & $\cdots$ & $e_{m_1+m_2-2}$\\
% 	$0$ & $\cdots$ & $0$ & $f_1$ & $f_2$ & $0$ & $\cdots$ & $0$	
% \end{tabular}\right].$$
% \end{Corollary}

%\subsection*{Acknowledgements}

%I am very grateful to Professor Toshiyuki Kobayashi for suggesting this problem to me. I also wish to thank Professors Michael Pevzner, Loïc Poulain d'Andecy and Victor Gayral for the fruitful discussions that drastically improved the quality of this paper.

\bibliographystyle{sigma}
\bibliography{biblio_SIGMA}
\LastPageEnding

\end{document}